\documentclass[11pt]{article}

\usepackage{amsmath, amssymb, amsthm, verbatim,enumerate,bbm,color} 
\usepackage{indentfirst}

\title{Packing, Counting and Covering  Hamilton cycles in random directed graphs}

\author{ Asaf Ferber
\thanks{Department of Mathematics, Yale University, and Department of Mathematics, MIT. Emails:
asaf.ferber@yale.edu, and ferbera@mit.edu.}\and Gal
Kronenberg\thanks{School of Mathematical Sciences, Raymond and
Beverly Sackler Faculty of Exact Sciences, Tel Aviv University, Tel
Aviv, 6997801, Israel. Email: galkrone@mail.tau.ac.il.} \and Eoin
Long\thanks{School of Mathematical Sciences, Raymond and Beverly
Sackler Faculty of Exact Sciences, Tel Aviv University, Tel Aviv,
6997801, Israel. Email: eoinlong@post.tau.ac.il.}}

\date{\today}
\allowdisplaybreaks

\theoremstyle{plain}
\newtheorem{theorem}{Theorem}[section]
\newtheorem{lemma}[theorem]{Lemma}
\newtheorem{claim}[theorem]{Claim}

\newtheorem{remark}[theorem]{Remark}
\newtheorem{definition}[theorem]{Definition}

\newcommand{\Bin}{\ensuremath{\textrm{Bin}}}

\newcommand{\whp}{w.h.p.\ }

\newcommand{\D}{\mathcal D}

\newcommand{\V}{\mathcal V}
\newcommand{\M}{\mathcal M}
\RequirePackage[normalem]{ulem} 
\RequirePackage{color}\definecolor{RED}{rgb}{1,0,0}\definecolor{BLUE}{rgb}{0,0,1} 
\providecommand{\DIFaddbegin}{} 
\providecommand{\DIFdelbegin}{} 
\providecommand{\DIFdelend}{} 

\begin{document}
\maketitle

\begin{abstract}
A Hamilton cycle in a digraph is a cycle that passes through all the
vertices, where all the arcs are oriented in the same direction. The
problem of finding Hamilton cycles in directed graphs is well
studied and is known to be hard. One of the main reasons for this,
is that there is no general tool for finding Hamilton cycles in
directed graphs comparable to the so called Pos\'a `rotation-extension'
technique for the undirected analogue. Let ${\mathcal D}(n,p)$ denote the
random digraph on vertex set $[n]$, obtained by adding each directed
edge independently with probability $p$. Here we
present a general and a very simple method, using known results, to attack problems of packing and counting Hamilton cycles in random directed
graphs, for every edge-probability $p>\log^C(n)/n$. Our results are
asymptotically optimal with respect to all parameters and apply equally
well to the undirected case.

\end{abstract}

%
%
%
%
%
%
%
%
%
%
%
%
%
%
%
%
%
%
%
%

\section{Introduction}

A \emph{Hamilton cycle} in a graph or a directed graph is a cycle
passing through every vertex of the graph exactly once, and a graph
is \emph{Hamiltonian} if it contains a Hamilton cycle. Hamiltonicity
is one of the most central notions in graph theory, and has been
intensively studied by numerous researchers in the last couple of
decades.

The decision problem of whether a given graph contains a Hamilton
cycle is known to be $\mathcal{NP}$-hard and is one of Karp's list
of 21 $\mathcal{NP}$-hard problems \cite{karp1972reducibility}.
Therefore, it is important to find general sufficient conditions for
Hamiltonicity and indeed, many interesting results were obtained in
this direction.

Once Hamiltonicity has been established for a graph
there are many questions of further interest. For example, the following are natural \DIFdelbegin  questions:
\begin{itemize}
    \item Let $G$ be a graph with minimum degree $\delta(G)$. Is it possible to find roughly $\delta(G)/2$ edge-disjoint Hamilton
    cycles? (This problem is referred to as the  \emph{packing} problem.)
     \item Let $\Delta(G)$ denote the maximum degree of $G$. Is it
    possible to find roughly $\Delta(G)/2$ Hamilton cycles for which
    every edge $e\in E(G)$ appears in at least one of these cycles? (This
     problem is referred to as the  \emph{covering} problem.)
     \item How many distinct Hamilton cycles does a given graph have?
    (This problem is referred to as the \emph{counting} problem.)
    \end{itemize}

All of the above questions have a long history and many results are known.
Let us define
$\mathcal G(n,p)$ to be the probability space of graphs on a vertex
set $[n]:=\{1,\ldots,n\}$, such that each possible (unordered) pair
$xy$ of elements of $[n]$ appears as an edge independently
with probability $p$. We say that a graph $G\sim \mathcal G(n,p)$
satisfies a property $\mathcal P$ of graphs with high
probability (w.h.p.) if the probability that $G$ satisfies
$\mathcal P$ tends to $1$ as $n$ tends to infinity.\\

{\bf Packing.} The question of packing in the probabilistic setting
was firstly discussed by Bollob\'as and Frieze in the 80's. They
showed in \cite{bollobas1985matchings} that if $\{G_i\}_{i=0}^{\binom {n}{2}}$ is a random graph process on $[n]$, where $G_0$ is the empty graph and $G_i$ is obtained from $G_{i-1}$ by adjoining a non-edge of $G_{i-1}$ uniformly at random, as soon as $G_i$ has minimum degree $k$ (where $k$
is a fixed integer), it has $\lfloor k/2\rfloor$ edge-disjoint
Hamilton cycles plus a disjoint perfect matching if $k$ is odd.
This
result generalizes an earlier result of Bollob\'as
\mbox{
\cite{bollobas1984evolution}
}
who proved (among other things) that
for $p=\frac{\ln n+\ln\ln n+\omega(1)}{n}$, a typical graph $G\sim
\mathcal G(n,p)$ is Hamiltonian. Note that this value of $p$ is optimal in
the sense that for $p=\frac{\ln n+\ln\ln n-\omega(1)}{n}$, it is
known that  w.h.p.
a graph $G\sim \mathcal G(n,p)$ satisfies
$\delta(G)\leq 1$, and therefore is not Hamiltonian. Later on,
Frieze and Krivelevich showed in \mbox{
\cite{frieze2008two}
}
that for
$p=(1+o(1))\frac{\ln n}n$, a graph $G\sim \mathcal G(n,p)$
w.h.p. contains $\lfloor\delta(G)/2\rfloor$ edge-disjoint Hamilton
cycles (in fact, this was proven only using pseudo-random hypothesis), which has afterwards been improved by Ben-Shimon,
Krivelevich and Sudakov in \cite{ben2011resilience} to $p\leq
1.02\frac{\ln n}n$. We remark that in this regime of $p$, w.h.p.
$G\sim \mathcal G(n,p)$ is quite far from being regular.
 As the culmination of a long line of research Knox,  K\"uhn and Osthus
\cite{knox2013edge}, Krivelevich and Samotij \mbox{
\cite{krivelevich2012optimal} }
and  K\"uhn and Osthus \cite{kuhn2014hamilton} completely solved
this question for the entire range of $p$.

For the non-random case, it is worth mentioning a recent remarkable
result due to Csaba, K\"uhn, Lo, Osthus and Treglown
\cite{csaba2013proof}
which
proved that for large enough $n$ and $d\geq \lfloor n/2\rfloor$,
every $d$-regular graph on $n$ vertices contains $\lfloor
d/2\rfloor$ edge-disjoint Hamilton cycles and one disjoint perfect
matching in case $d$ is odd. This result settles a long standing
problem due to Nash-Williams \mbox{
\cite{nash1970hamiltonian}
}
for large
graphs.
\\

{\bf Covering.} The problem of covering the edges of a random graph
was firstly studied in \cite{glebov2014covering} by Glebov,
Krivelevich and Szab\'o. It is shown that for $p\geq
n^{-1+\varepsilon}$, the edges of a typical $G\sim \mathcal G(n,p)$
can be covered by $(1+o(1))np/2$ edge-disjoint Hamilton cycles. Furthermore they proved  analogous results also in the pseudo-random setting. In
\cite{hefetz2014optimal}, Hefetz, Lapinskas, K\"uhn and Osthus
improved it by showing that for some $C>0$ and $\frac{\log ^C(n)}{n}
\leq p \leq 1-n^{-1/8}$, one can cover all the edges of a typical
graph $G\sim \mathcal G(n,p)$ with $\lceil\Delta(G)/2\rceil$
Hamilton cycles.
\\

\DIFdelend {\bf Counting.} Given a graph $G$, let $h(G)$ denote the number of
distinct Hamilton cycles in $G$. Strengthening the classical theorem
of Dirac from the 50's \cite{dirac1952some}, S\'ark\"ozy, Selkow and
Szemer\'edi \cite{sarkozy2003number} proved that every graph $G$ on
$n$ vertices with minimum degree at least $n/2$ contains not only
one but at least $c^nn!$ Hamilton cycles for some small positive
constant $c$. They also conjectured that this  $c$ could be improved to
$1/2-o(1)$. This was later  proven by Cuckler and Kahn
\cite{cuckler2009hamiltonian}.
In fact, Cuckler and Kahn proved a
stronger result: every graph $G$ on $n$ vertices with minimum degree
$\delta(G)\geq n/2$ has $h(G)\geq
\left(\frac{\delta(G)}{e}\right)^{n}(1-o(1))^{n}$. A typical random
graph $G\sim \mathcal G(n,p)$ with $p>1/2$ shows that this estimate
is sharp (up to the $(1-o(1))^n$ factor). Indeed, in this case with
high probability $\delta(G)=pn+o(n)$ and the expected number of
Hamilton cycles is $p^n(n-1)!<(pn/e)^n$.

In the  random/pseudo-random setting,
building on ideas of Krivelevich \cite{krivelevich2012number}, in
\cite{glebov2013number} Glebov and Krivelevich showed that
for $p\geq \frac{\ln n+\ln\ln n+\omega(1)}n$ and for a typical
$G\sim \mathcal G(n,p)$ we have $h(G)=(1-o(1))^nn!p^n$. That is, the
number of Hamilton cycles is, up to a sub-exponential factor,
concentrated around its mean. For larger values of $p$, Janson
showed
\mbox{
    \cite{janson1994numbers}
}
that the distribution of $h(G)$ is log-normal, for $G\sim \mathcal
G(n,p)$ with $p=\omega(n^{-1/2})$.\\

In this paper we treat the three of these problems in the random directed setting.
A \emph{directed graph} (or \emph{digraph}) is a pair $D=(V,E)$ with a set of \emph{vertices} $V$ and a
set of \emph{arcs} $E$, where each arc is an ordered pair of
elements of $V$. A directed graph is called \textit{oriented}, if for every pair of vertices $u,v\in V$, at most one of the directed edges $\overrightarrow{uv}$ or $\overrightarrow{vu}$ appears in the graph. A \textit{tournament} is an oriented complete graph. A \emph{Hamilton cycle} in a digraph  is a cycle going through all the vertices exactly once, where all the arcs are oriented in the same direction in a cyclic order. Given a directed graph $D$ and a vertex $v\in V$, we let $d_D^+(v)$ and $d^-_D(v)$ denote its out- and in- degree in
$D$.

Let $\mathcal D(n,p)$ be the
probability space consisting of all directed graphs on vertex set
$[n]$ in which each possible arc is added with probability
$p$ independently at random. The problem of determining the range
of values of $p$ for which a typical graph $D\sim \mathcal D(n,p)$ is
Hamiltonian goes back to
the early 80's, where McDiarmid \cite{mcdiarmid1980clutter} showed,
among other things, that an elegant coupling argument gives the inequality
$$\Pr[G\sim \mathcal G(n,p) \textrm{ is
    Hamiltonian}]\leq \Pr[D\sim \mathcal D(n,p) \textrm{ is
    Hamiltonian}].$$ Combined with the result of Bollob\'as
\cite{bollobas1984evolution} it follows that a typical $D\sim
\mathcal D(n,p)$ is Hamiltonian for $p\geq \frac{\ln n+\ln\ln
    n+\omega(1)}n$. Later on, Frieze showed in
\cite{frieze1988algorithm} that the same conclusion holds for $p\geq
\frac{\ln n+\omega(1)}n$. The result of Frieze is optimal in the
sense that for $p=\frac{\ln n-\omega(1)}n$, it is not difficult to
see that for a typical $D\sim \mathcal D(n,p)$ we have
$\min _{v\in V} \{\delta ^+(v), \delta ^-(v)\} =0$ and therefore $D$ is not Hamiltonian. Robustness
of Hamilton cycles in random digraphs was studied by Hefetz, Steger
and Sudakov in \cite{robustlyHamDigraphs} and by Ferber, Nenadov,
Noever, Peter and Skori\'c in \cite{ferberrobust}.

\subsection{Our results}

While
in general/random/pseudo-random graphs there are many known
results, much less is known about the problems of counting, packing
and covering in the directed setting. The main difficulty is that in
this setting the so called Pos\'a rotation-extension technique
(see \cite{posa}) does not work in its simplest form.

In this paper we present a simple method to attack and
approximately solve all the above mentioned problems in
random/pseudo-random directed graphs, with an optimal (up to a
$polylog(n)$ factor) density. Our method is also applicable in the
undirected setting, and therefore reproves many of the above
mentioned results in a simpler way.

    The problem of packing Hamilton cycles in digraphs goes back to the 70's.
    Tilson \cite{tillson1980hamiltonian} showed that every complete digraph has a Hamilton decomposition. Recently, a remarkable result of  K{\"u}hn and Osthus  (see
\cite{kuhn2013hamilton})  proves  that  for any regular orientation of a sufficiently dense graph one can find a Hamilton decomposition.
    In the case of a random directed graph, not much is known regarding packing Hamilton cycles. Our first result proves the  existence of $(1-o(1))np$
    edge-disjoint Hamilton cycles \DIFaddbegin in $\D(n,p)$.
    \begin{theorem}
        \label{thm:PackingRandom}
        For $p= \omega \left(\frac {\log^4n}{n}\right)$, w.h.p.\ the digraph
        $D\sim \mathcal D(n,p)$ has $(1-o(1))np$ edge-disjoint Hamilton cycles.
    \end{theorem}

We also show that in random directed graphs one can cover all the edges by not too many cycles.

    \begin{theorem}
    	\label{thm:CoverRandom}
        Let $p=\omega\left(\frac{\log^2n}n\right)$. Then, a digraph $D\sim
        \mathcal D(n,p)$ w.h.p.\  can be covered  with $(1+o(1))np$ directed
        Hamilton cycles.
    \end{theorem}

    The problem of counting Hamilton cycles in digraphs was already studied in the early 70's by Wright in \cite{wright1973many}.
    However, counting Hamilton cycles in tournaments is an even older problem which goes back to one of the first applications of the probabilistic method by Szele
\cite{szele1943kombinatorikai}. He proved that there
    are tournaments on n vertices with at least $(n-1)!/2^n$
    Hamilton cycles.
    Thomassen
\cite{thomassen1985hamilton}
conjectured that in fact every regular tournament contains at least
    $n^{(1-o(1))n}$ Hamilton cycles.
    This conjecture  was solved by Cuckler
\cite{cuckler2007hamiltonian} who proved that every regular
tournament on $n$ vertices contains at least $\frac
{n!}{(2+o(1))^n}$ Hamilton cycles.  Ferber,  Krivelevich
and  Sudakov \cite{ferber2012counting} later extended Cuckler's result for
every nearly $cn$-regular oriented graph for $c>3/8$.  Here, we
count the number of Hamilton cycles in random directed graphs and
improve a  result of Frieze and Suen from \cite{frieze1992counting}.
    We show that the number of directed Hamilton cycles in such
    random graphs is concentrated (up to a
    sub-exponential factor) around its mean.

    \begin{theorem}
        \label{thm:CountingRandom}
        Let $p=\omega\left(\frac{\log^2n}n\right)$. Then, a digraph $D\sim
        \mathcal D(n,p)$ w.h.p.\ contains $(1\pm o(1))^nn!p^n$
        directed Hamilton cycles.
    \end{theorem}

Finally, the same proof method can be used to prove analogous results when instead working with pseudo-random graphs. We direct the reader to Definition \ref{definition:pseudorandom} in Section \ref{sec:HamInPseudo} for the notion of pseudo-randomness used here. The following theorems show that at a cost of an additional $\text{polylog} n$ factor in the density we obtain analogues of Theorem \ref{thm:PackingRandom}, \ref{thm:CoverRandom}, \ref{thm:CountingRandom} for pseudo-random digraphs. Below we will write $o_{\lambda}(1)$ for some quantity tending to $0$ as $\lambda \to 0$.

    \begin{theorem}
        \label{thm:PackingPseudoRandom}
        Let $D$ be a $(n,\lambda,p)$ pseudo-random digraph where
        $p=\omega\left(\frac {\log^{14}n}n\right)$. Then $D$ contains
        $(1-o_{\lambda}(1))np$ edge-disjoint Hamilton cycles.
    \end{theorem}

	 \begin{theorem}
    	\label{thm:CoverPseudoRandom}
        Let $D$ be a $(n,\lambda,p)$ pseudo-random digraph where
        $p=\omega\left(\frac {\log^{14}n}n\right)$. Then $D$
        can be covered  with $(1+o_{\lambda}(1))np$ directed Hamilton cycles.
    \end{theorem}

    \begin{theorem}
        \label{thm:CountingPseudoRandom}
         Let $D$ be a $(n,\lambda,p)$ pseudo-random digraph where
        $p=\omega\left(\frac {\log^{14}n}n\right)$. Then $D$
        can be contains $(1 - o_{\lambda }(1))^nn!p^n$ directed Hamilton cycles.
    \end{theorem}

    We have only included the proof of Theorem \ref{thm:PackingPseudoRandom}
    which modifies the proof of Theorem \ref{thm:PackingRandom} to
    the pseudo-random setting. The other results can be proven in a similar
    manner (these other proofs are in fact slightly easier).

 \begin{remark}
    We also draw attention to the fact that all of our proofs also apply to
    $\mathcal G(n,p)$ with the same probability thresholds as in Theorem
    \ref{thm:PackingRandom}, \ref{thm:CoverRandom} and \ref{thm:CountingRandom}. Although all these results are known in $\mathcal G(n,p)$ (and in fact even much more), our approach provides us with short and elegant proofs. For convenience, we state the exact statements which follow from our proofs:
    	\begin{itemize}
    		\item For $p = \omega \big ( \frac {\log ^4 n }{n} \big )$ our
    		approach gives that $G \sim \mathcal G(n,p)$ whp contains
    		$(1-o(1))np/2$
    		edge disjoint Hamitlon cycles. As mentioned in the packing section
    		above, here it is known that for all $p$ whp
    		${G} \sim \mathcal G(n,p)$ contains
    		$\lfloor \delta ({G})/2 \rfloor$ edge disjoint
    		Hamilton cycles (see \cite{knox2013edge},
    		\cite{krivelevich2012optimal} and \cite{kuhn2014hamilton}).
    		\item For $p = \omega \big ( \frac {\log ^2 n }{n} \big )$ our
    		approach gives that $G \sim \mathcal G(n,p)$ whp
    		contains $(1+o(1))np/2$
    		Hamilton cycles covering all edges of $G$. As
    		mentioned in the covering section above, here it is known
    		that there is some constant $C>0$ such that for
    		$\frac {\log ^C n}{n} \leq p \leq 1 - n^{-1/18}$ whp
    		${G} \sim \mathcal G(n,p)$ has an edge covering with
    		$\lceil \Delta ({G})/2 \rceil$ Hamilton cycles
    		(see \cite{hefetz2014optimal}).
    		\item For $p = \omega \big ( \frac {\log ^2 n }{n} \big )$ our
    		approach gives that $G \sim \mathcal G(n,p)$ whp contains
    		$(1\pm o(1))^nn!p^n$ Hamilton cycles. As
    		mentioned in the counting section above, here it is known
    		that such a bound already applies for $p>
    		\frac{ \log n + \log \log n + \omega (1) }{n}$.
    	\end{itemize}
\end{remark}

\subsection{Notation and terminology}
We denote by $D_n$ the \emph{complete} directed
graph on $n$ vertices (that is, all the possible $n(n-1)$ arcs
appear), and by $D_{n,m}$ the complete bipartite digraph with parts
$[n]$ and $[m]$. Given a directed graph $F$ and a vector $\bar{p}\in
(0,1]^{E(F)}$, we let $\mathcal D(F,\bar{p})$ denote the probability
space of sub-digraphs $D$ of $F$, where for each arc $e\in E(F)$, we
add $e$ into $E(D)$ with probability $p_e$, independently at random.
In the special case where $p_e=p$ for all $e$, we simply denote it
by $\mathcal D(F,p)$. In the case where $F=D_n$, we write $\mathcal
D(n,p)$ and in the case $F=D_{n,m}$ we write $\mathcal D(n,m,p)$. Given
a digraph $D$ and two sets
$X,Y \subset V(D)$ we write $E_D(X,Y) = \{ \overrightarrow{xy} \in E(D): x \in X, y\in Y\}$. Also let $e_{D}(X,Y) = |E_D(X,Y)|$ and $e_{D}(X) = |E_D(X,X)|$. We will also occasionally
make use of the same notation for graphs $G$, i.e. $e_G(X,Y)$.
For a vertex $v$ we denote $N_D^+(v)=E_D(\{v\},V(D))$ and $N_D^-(v)=E_D(V(D),\{v\})$. Let $d_D^+(v)=|N_D^+(v)|$ and $d_D^-(v)=|N_D^-(v)|$.
Lastly, we write $x\in a\pm b$ to mean that $x$ is in the interval $[a-b,a+b]$.

%
%
%
%
%
%
%
%
%
%
%
%
%
%
%
%
%
%
%
%

\section{Overview and auxiliary results}

\subsection{Proof overview}

Our aim in this subsection is to provide an overview of the proofs of Theorems
\ref{thm:PackingRandom}, \ref{thm:CoverRandom} and \ref{thm:CountingRandom}. In particular, we hope to highlight the similarities and differences which occur for the packing, counting and covering problems. To do this, we will first describe an approach to solve similar problems for a more restricted model of random digraph. We then outline how these results can be used to solve the corresponding problems for ${\cal D}(n,p)$.

Suppose that we are given a partition $[n] = V_0 \cup V_1 \cup \cdots \cup V_{\ell }$, with $|V_0| = s$ and $|V_j| = m$ for all $j\in [\ell ]$ so that $n = m \ell  + s$ (here $s= \omega(m)$ and $\ell=\text{polylog}(n)$). Consider the following way to select random digraph $F$:
	\begin{enumerate}
		\item For all $j\in [\ell -1]$, directed edges from  $V_j$ and
		$V_{j+1}$ are adjoined to $F$ with probability
		$p_{in}$ independently. Let $F_j$ denote
		this sub-digraph of $F$;
		\item The directed edges ({\emph a}) in $V_0$  ({\emph b})
		from $V_0$ to $V_1$ ({\emph c}) from 	$V_{\ell }$ to $V_0$
		and ({\emph d}) from $V_{\ell }$ to $V_1$ are adjoined to $F$ with probability
		$p_{ex}$ independently. Let $F _0$ denote
		this subdigraph of $F$.
	\end{enumerate}
This selection process gives a distribution on a set of digraphs. We will write $\mathcal F$ to denote this distribution, and write $F \sim {\cal F}$ to denote a digraph $F$ chosen according to it.
We will describe how to show that if $F \sim {\cal F}$ then whp, for appropriate values of $p_{in}$ and $p_{ex}$, we have:
	\begin{enumerate}[(i)*]
		\item $(1-o(1))mp_{in}$ edge disjoint Hamilton cycles which contain
		almost all edges of type $1$ in $F$;
		\item $(1 + o(1))mp_{in}$ Hamilton cycles which cover all edges of type
		$1$ in $F$;
		\item $(1-o(1))^{n-s}(m!)^{\ell -1} p_{in}^{n-s-m}$ directed Hamilton cycles in
		$F$.
	\end{enumerate}

To do this, we first expose edges of type $1.$ above. Using known matching results, for $p_{in} =\omega ( \log ^C m/m)$ and $\ell \leq m$ say, it can be shown that whp for every $j\in [\ell -1]$:
	\begin{enumerate}[(i)]
		\item $F_j$ contains $L_{pack} := (1 - o(1))mp_{in}$ edge disjoint perfect matchings, $\{M_i^j\}_{i=1}^{L_{pack}}$;
		\item $F_j$ contains $L_{cov}: =(1+o(1))mp_{in}$ perfect matchings covering all edges of $F_j$, $\{M_i^j\}_{i=1}^{L_{cov}}$;
		\item $F_j$ contains $(1-o(1))^m m!p_{in}^m$ perfect matchings.
	\end{enumerate}
Now note that in (i), (ii) and (iii) above, by combining a perfect matching from each ${F}_j$ for each $j\in [\ell -1]$ we obtain a collection of $m$ vertex disjoint directed paths from $V_1$ to $V_{\ell }$, covering $\bigcup _{j\in [\ell]} V_{j}$. We refer to such a collection of paths ${\cal P}$ as a \emph{matching path system}. 
	\begin{enumerate}[(i)]
		\item For each $i\in [L_{pack}]$, by combining the disjoint matchings $\{M_i^j\}_{j=1}^{\ell-1}$ from (i) in this way,
		we obtain a matching path system $\mathcal P_i$. This gives $L_{pack}$ 
		edge disjoint matching path systems ${\cal P}_1 ,\ldots ,
		{\cal P}_{L_{pack}}$.
		\item For each $i\in [L_{cov}]$, by combining the matchings 
		$\{M_i^j\}_{j=1}^{\ell-1}$ from (ii) in 
		this way we obtain a matching path system $\mathcal P_i$. This gives 
		$L_{cov}$ matching path systems	${\cal P}_1 ,\ldots , {\cal P}_{L_{cov}}$, 
		which cover all edges in the digraphs $F_j$ for $j\in [\ell -1]$.
		\item Lastly, by choosing different matching between
		the partitions from (iii), we have many choices for how to
		build our matching path system ${\cal P}$. We obtain at least
		$(1-o(1))^{m\ell } (m!)^{\ell -1} p_{in}^{m(\ell -1)} \geq (1-o(1))^{n}
		(m!)^{\ell -1} p_{in}^{n-s-m}$ such choices for ${\cal P}$.
	\end{enumerate}

Now let ${\cal P} = \{P_1,\ldots ,P_m\}$ be a fixed matching path system. Assume that each $P_i$ begins at a vertex $s_i \in V_1$ and terminates at a vertex $t_i \in V_{\ell }$. These vertices are distinct by construction. We will now describe how to include all paths in ${\cal P}$ into a directed Hamilton cycle. To do this simply contract each directed path $P_i$ to single vertex which we also denote by $P_i$. Now expose the edges of type 2. above and view them as edges of a random digraph on vertex set ${\widetilde V} = V_0 \cup \{P_1,\ldots ,P_m\}$. Note that 
the following edges all appear with probability $p_{ex}$:	
\begin{itemize}
		\item All directed edges in $V_0$. These come from edges of type 2. 
		({\emph a}) 
		above;
		\item Directed edges from $V_0$ to $\{P_1,\ldots ,P_m\}$ and 
		from $\{P_1,\ldots ,P_m\}$ to $V_0$. These come respectively from 
		edges of type 2. ({\emph b}) and ({\emph c}) above;
		\item Directed edges within the set $\{P_1,\ldots ,P_m\}$. These
		come from edges of type 2 ({\emph d}) above. (Here we may obtain 
		a loop on the vertices $P_i$, which we simply ignore.)
\end{itemize}
As all such edges appear independently, the resulting random digraph is distributed identically to ${\cal D}(s + m, p_{ex})$. 
By known Hamiltonicity result for ${\cal D}(n,p)$, provided that $p_{ex} =\omega\left( \log ^C (m+s)/(m+s)\right)$ we obtain that this digraph is Hamiltonian with very high probability. However, it is easy to see that by construction  a directed Hamilton cycle in this contracted digraph pulls back to a directed Hamilton cycle in $F$, which contains the paths in ${\cal P}$ as directed subpaths. Thus we have shown how to turn a single matching path system into a Hamilton cycle.

Now in the case of (ii)*, we can complete each of the matching path systems ${\cal P}_{1},\ldots {\cal P}_{L_{cov}}$ into Hamilton cycles by using edges of type 2. described above. This can also be used to show whp many of the matching paths systems from (iii)* complete to (distinct) directed Hamilton cycles. However, to pack the Hamilton cycles in the case of (i)* more care must be taken as we cannot use the same edges twice. To get around this, we distribute the edges of type 2. to create an individual random digraph for each ${\cal P}_i$. Provided that $p_{ex}$ is sufficiently large (and $m,\ell$ and $s$ are carefully chosen) each of these individual random digraphs will be Hamiltonian whp. This completes the description of (i)*, (ii)* and (iii)* above.

Now our approach for dealing with the packing, covering and counting problems on ${\cal D}(n,p)$ is to show that with high probability we can break $D \sim {\cal D}(n,p)$ into subdigraphs distributed similarly to $F$ above. However the type of decomposition chosen is again dependent on the problem at hand. With the packing it is important that these graphs are edge disjoint. With the covering, it will be important every edge of ${\cal D}(n,p)$ appears as an edge of type 1. in one of these digraphs (recall these were the only edges guaranteed to be covered in (ii)*). The counting argument is less sensitive, and simply work with many such digraphs. Dependent on the problem, we can apply our strategy above for $F \sim {\cal F}$ to each of these digraphs separately. Combining the resulting Hamilton cycles from either (i)*, (ii)* or (iii)* in each of these digraphs will then solve the corresponding problem for ${\cal D}(n,p)$.

\subsection{Probabilistic tools}

We will need to employ bounds on large deviations of random
variables. We will mostly use the following well-known bound on the
lower and the upper tails of the binomial distribution due to
Chernoff (see \cite{alon2004probabilistic}, \cite{janson2011random}).

\begin{lemma}[Chernoff's inequality]
    \label{Chernoff}
    Let $X \sim \operatorname{Bin} (n, p)$ and let
    $\mu = \mathbb{E}(X)$. Then
    \begin{itemize}
        \item $\Pr[X < (1 - a)\mu] < e^{-a^2\mu/2}$ for every $a > 0$;
        \item $\Pr[X > (1 + a)\mu] < e^{-a^2\mu/3}$ for every $0 < a < 3/2$.
    \end{itemize}
\end{lemma}
\begin{remark}\label{CheHyper} The conclusions of Chernoff's inequality remain the same
    when $X$ has the hypergeometric distribution (see \cite{janson2011random}, Theorem~2.10).
\end{remark}

We will also find the following bound useful.

\begin{lemma}\label{Che}
    Let $X\sim \Bin(n,p)$. Then $\Pr\left[X\geq k\right]\leq \left(\frac{enp}{k}\right)^k.$
\end{lemma}

\begin{proof} Just note that
    $$\Pr\left[X\geq k\right]\leq \binom{n}{k}p^k\leq
    \left(\frac{enp}{k}\right)^k.$$
\end{proof}

\subsection{Perfect matchings in bipartite graphs and random bipartite graphs}

The following lower bound on the number of perfect matchings in an
$r$-regular bipartite graph is also known as the Van der Waerden
conjecture and has been proven by Egorychev \cite{egorychev1981solution} and by
Falikman \cite{falikman1981proof}:

\begin{theorem} \label{VanDerWaerden}
    Let $G=(A\cup B,E)$ be an $r$-regular bipartite graph with parts of
    sizes $|A|=|B|=n$. Then, the number of perfect matchings in $G$ is
    at least $\left(\frac{r}{n}\right)^nn!$.
\end{theorem}

The following lemma  is an easy corollary of the so called Gale-Ryser theorem (see, e.g. \cite{lovaszproblems}).
\begin{lemma}(Lemma 2.4, \cite{ferber2014packing})\label{lemma:many-matchings}
    Let $G$ is a random bipartite graph between two vertex sets both of
    size $n$, where edges are chosen independently with probability
    $p = \omega (\log n / n)$. Then with probability $1 - o(1/n)$ the graph
    $G$ contains $(1- o(1))np$ edge disjoint perfect matchings.
\end{lemma}

\subsection{Converting paths into Hamilton cycles}\label{sec:convertPathToHam}

The following definitions will be convenient in our proofs.

\begin{definition} Suppose that $X$ is a set of size $n$ and that $\ell ,m,s$ are positive integers with $n = m\ell +s$. A sequence ${\mathcal V} = (V_0, V_1,\ldots,V_{\ell})$ of subsets of $X$ is called an $(\ell,s)$-partition of $X$ if
    \begin{itemize}
        \item $X=V_0 \cup V_1\cup\ldots \cup V_{\ell}$ is a partition of $X$, and
        \item $|V_0|=s$, and
        \item $|V_i| = m$ for every $i\in [\ell ]$.
    \end{itemize}
\end{definition}

\begin{definition}  \label{Dn}   Given an $(\ell,s)$-partition
    ${\mathcal V} = (V_0, V_1,\ldots,V_{\ell})$ of a set $X$, let
    $D_n({\mathcal V})$ denote the digraph on
    vertex set $X=[n]$ consisting of all edges $\overrightarrow{uv}$
    such that:
    \begin{enumerate}
        \item $u\in V_j$, $v \in V_{j+1}$ for some $j\in [\ell-1]$, or
        \item $u\in V_0$ and $v\in V_0 \cup V_1$, or
        \item $u \in V_{\ell}$, $v \in V_{0} \cup V_1$.
    \end{enumerate}
    We call edges of type 1. \emph{interior edges} and call edges of
    type 2. and  3. \emph{exterior edges}.
\end{definition}

Suppose that we are given two disjoint sets $V$ and $W$ and a
digraph $D$ on vertex set $V \cup W$. Suppose also that
we have $m$ disjoint ordered pairs ${\mathcal M} = \{(w_i,x_i): i\in
[m]\} \subset W\times W$. Then we define the following auxiliary
graph.

\begin{definition}\label{def:auxGraph}

    Let $D({\mathcal M},V)$ denote the following auxiliary digraph on vertex set ${\mathcal M} \cup V$ where $\M = \{u_1,\ldots ,u_m\}$ and each $u_i$ refers to the pair $(w_i,x_i)$. Then given any two vertices $v_1,v_2 \in V$, we have:
    \begin{itemize}
        \item $\overrightarrow{v_1v_2}$ is an edge in
        $D({\mathcal M},V)$ if it appears in
        $D$;
        \item $\overrightarrow{v_1u_i}$ is an edge
        in $D({\mathcal M},V)$ if $\overrightarrow{v_1w_i}$
        is an edge in ${D}$;
        \item $\overrightarrow{u_iv_1}$ is an edge
        in $D({\mathcal M},V)$ if $\overrightarrow{x_iv_1}$
        is an edge in ${D}$;
        \item $\overrightarrow{u_iu_j}$ is an edge
        in $D({\mathcal M},V)$ if $\overrightarrow{x_iw_j}$
        is an edge in $D$.
    \end{itemize}
\end{definition}

\begin{remark} \label{remark: Ham cycles correspond to Ham cycles} Note that if $D({\mathcal M},V)$ contains a directed Hamilton cycle and $W$ can be decomposed into vertex disjoint directed $w_ix_i$-paths for all $i\in [m]$
(paths starting at $w_i$ and ending at $x_i$) then $D$ contains a directed Hamilton cycle.
\end{remark}

%
%
%
%
%
%
%
%
%
%
%
%
%
%
%
%
%
%
%
%

\section{Counting Hamilton cycles in $\D(n,p)$}

In this section we prove Theorem \ref{thm:CountingRandom}. The proof of this theorem is relatively simple and contains most of the ideas for the other main results and therefore serves as a nice warmup. 

\begin{proof} We will first prove the upper bound. For this, let $X_H$ denote the random variable that counts the
    number of Hamilton cycles in $D\sim \mathcal D(n,p)$. It is clear
    that $\mathbb{E}[X_H]=(n-1)!p^n$. By Markov's inequality, we therefore
    have $$Pr(X_H\geq (1+o(1))^nn!p^n) \leq
    \frac {{\mathbb E}[X_H]}{(1+o(1))^nn!p^n} = (1-o(1))^n = o(1).$$
    Thus $X_H\leq (1+o(1))^nn!p^n$ w.h.p..

    We now prove the lower bound, i.e. $X_H\geq (1-o(1))^nn!p^n$ w.h.p..
    Let $\alpha:=\alpha(n)$ be a function tending to infinity arbitrarily
    slowly with $n$. We prove the lower bound on $X_H$ under the assumption
    that $p \geq \alpha ^2 \log ^2 n/ n$. Let us take $s$ and $\ell $
    to be integers
    where $s$ is roughly $\frac{n}{\alpha \log n}$ and $\ell$ is roughly
    $2\alpha\log n$ and there is an integer $m$ with $n = \ell m + s$.
    Also fix a set $S\subseteq V(G)$ of order $s$ and let us set $V' = V({D}) \setminus
    S$. The set $S$ will be used to turn collections of vertex disjoint paths into Hamilton cycles. 

    To begin, take a fixed $(\ell , s)$-partition ${\cal V} = (V_0,
    V_1,\ldots ,V_{\ell })$ with $V_0 = S$. We claim the following:
	\vspace{1mm}

	\noindent \textbf{Claim:} Given ${\cal V}$ as above, taking $D \sim {\cal D}(n,p)$, the random digraph
	$D \cap D_n({\cal V})$ (where $D_n({\mathcal V})$ is as in Definition \ref{Dn}) contains at least
	$(1-o(1))^n m!^{\ell -1 }p^{m(\ell -1)}$ distinct Hamilton cycles with
	probability $1 - o(1)$.
	\vspace{1mm}	
	
    To see this, first expose the interior edges of $D \cap D_n({\cal V})$.
    For each $j\in [\ell -1]$ let $F_{j}: = E_D(V_j,V_{j+1})$. Observe that
    $F_j\sim D(m,m,p)$. It will be convenient for us
    to view $F_j$ as a bipartite graph obtained by ignoring the
    edge directions. Since
    $p=\omega\left(\frac{\log n}{m}\right)$, by Lemma
    \ref{lemma:many-matchings} with probability $1-o(1/n)$ we
    conclude that $F_j$ contains
    $(1-o(1))mp$ edge-disjoint perfect matchings. Taking a union bound over
    all $j\in [\ell -1]$ we
    find that whp $F_j$ contains a $(1-o(1))mp$-regular subgraph for all
    $j\in [\ell -1]$.

    Apply Theorem \ref{VanDerWaerden} to each of these subgraphs. This give
    that for each $j\in [\ell -1]$ the graph $F_j$ contains at least
    $(1-o(1))^{m}m !p^{m}$ perfect matchings.
    Combining a perfect matching from each of the $F_j$'s we obtain
    a family $\mathcal P$ of $m$ vertex disjoint paths
    which spans $V'$. Let ${\Lambda }_{\cal V}$ denote the set of
    all such ${\cal P}$. From the choices of perfect matchings in
    each $F_j$ we obtain that whp 
    \begin{align}
    	\label{equation: size of Lambda set control}
       |{\Lambda }_{\cal V}| \geq
        \left((1-o(1))^{m} m!p^{m}\right)^{\ell - 1}=
        (1-o(1))^{n}\left(m!\right)^{\ell -1} p^{n-s}.
    \end{align}
    Now let ${\cal P} = \{P_1,\ldots ,P_m\} \in {\Lambda }_{\cal V}$. Let
    $${\mathcal M} = \{(u_{i},v_{i}) \in V_1\times V_{\ell }: P_{i}
    \mbox{ is a }u_{i}-v_{i} \mbox{ directed path}\}.$$
    Let us consider the  auxiliary digraph
    $D(\M,V_0)$ as in Definition \ref{def:auxGraph}. As we expose the
    exterior edges of $D$ in $D_n({\cal V})$ it is easy to see that
    $D(\M ,V_0)\sim \mathcal D(s+m,p)$. Furthermore, a Hamilton cycle in
    $D(\M,V_0)$ gives a Hamilton cycle in $D$ by Remark
    \ref{remark: Ham cycles correspond to Ham cycles}. However, it is
     well-known digraphs in ${\cal D}(n,p')$
    are Hamiltonian  w.h.p. for say
    $p' > 2\log n /n$ (\cite{frieze1988algorithm}). Since
    $$p= \frac {\alpha ^2 \log ^2 n }{n}
    \geq  \frac {\alpha \log n }{s + m}
    = \omega \left ( \frac{\log (s+m)}{s+m} \right )$$ we find that
    $D(\M,V_0)$ is Hamiltonian w.h.p.. Thus we have shown that
    	\begin{equation}
    		\label{equation: extension to Hamitlon cycle probability}
    		Pr \Big ({\cal P} \mbox { does not extend to a Hamilton cycle in }
    		{D} \cap D_n({\cal V}) \Big ) = o(1).
    	\end{equation}
    Let ${\Lambda }_{\cal V} ' \subset {\Lambda }_{\cal V}$ denote the set of
    ${\cal P} \in {\Lambda }_{\cal V}$ which do not extend to a Hamilton
    cycle in $D \cap D_n({\cal V})$. By
    \eqref{equation: extension to Hamitlon cycle probability} we have
    ${\mathbb E}(|{\Lambda }_{\cal V} '|) =
    o(|{\Lambda }_{\cal V}|)$. Using Markov's inequality we obtain
    that $|{\Lambda }_{\cal V} '| = o(|{\Lambda }_{\cal V}|)$ whp. Combined
    with \eqref{equation: size of Lambda set control} this gives that
    $|{\Lambda }_{\cal V} \setminus {\Lambda }_{\cal V} '|
    \geq (1-o(1))^{n}\left(m!\right)^{\ell -1} p^{n-s}$ whp. Lastly,
    to complete the proof of the claim, note that any two distinct families
    $\mathcal P, \mathcal P'
    \in {\Lambda }_{\cal V} \setminus {\Lambda }_{\cal V} '$ yield
    different Hamilton cycles -- indeed, by deleting the vertices of
    $S$ from the Hamilton cycle it is easy to recover the paths ${\cal P}$.
    This proves the claim.
	
	Now to complete the proof of the theorem, let $\Gamma $ denote the set
	of $(\ell ,s)$-partitions ${\cal V}$ with $V_0 =S$ which satisfy the
	statement of the claim. By Markov's inequality we have $|\Gamma |
	\geq (1-o(1)) \frac {(n-s)!}{(m!)^{\ell }}$ whp. Since for distinct
	${\cal V}, {\cal V}' \in {\Gamma }$ the Hamilton cycles in $D
	\cap D_n({\cal V})$ are all distinct, we find that whp $D$ contains
	at least
	\begin{align*}
        |\Gamma | (1-o(1))^{n} \left(m !\right)^{\ell - 1} p^{n-s}
			& \geq (1-o(1))^{n}\frac {(n-s)!}{(m!)^{\ell }}(m!)^{\ell -1}p^{n} \\       	
        	& =(1-o(1))^{n}\frac {(n-s)!}{m!}p^{n}
        	=(1-o(1))^n n!p^n
    \end{align*}
   distinct Hamilton cycles. The final equality here holds since
   $m < n/ \alpha \log n $ gives that $m! < e^{n/\alpha }$ and
   $(n)_s \leq n^s =(1+o(1))^n$ since $s = o(n/\log n)$. This completes the proof of the
   theorem.
\end{proof}

%
%
%
%
%
%
%
%
%
%
%
%
%
%
%
%
%
%
%
%

\section{Packing Hamilton cycles in ${\cal D}(n,p)$}

In this section we prove Theorem \ref{thm:PackingRandom}. The heart of the argument is contained in the following lemma.

\begin{lemma}
    \label{lemma: building cycles in (l,s)-partitions}
    Let ${\mathcal V} = (V_0, V_1,\ldots,V_{\ell})$ be an
    $(\ell ,s)$-partition of a set $X$ of size $n = \ell m +s$.
    Suppose that we select a random subdigraph $F$ of
    $D_n({\mathcal V})$ as follows:
    	\begin{itemize}
    		\item include each interior directed
    		edge of $D_n({\cal V})$ independently with probability $p_{in}$;
    		\item include each exterior directed
    			edge of $D_n({\cal V})$ independently with probability $p_{ex}$.
    	\end{itemize}
    Then, provided $p_{in}= \omega( \log n/m)$ and
    $p_{ex} =\omega(  mp_{in} \log n/(m+s))$,  w.h.p.\ $F$ contains
    $(1-o(1))mp_{in}$ edge-disjoint Hamilton cycles.
\end{lemma}

\begin{proof}
    We begin by exposing the interior edges of $F$.
    For $j\in [\ell-1]$ all
    edges $E_{D_n}(V_j,V_{j+1})$ appear in $E_F(V_j,V_{j+1})$
    independently with probability $p_{in}$. By ignoring the
    orientations, we can view $E_F(V_j,V_{j+1})$ as a bipartite graph.
    From Lemma \ref{lemma:many-matchings}, since $p_{in} =\omega( \log m/m)$
    we find that \whp for all $j\in [\ell-1]$ the graph 
    $E_F(V_j,V_{j+1})$ contains $L:=(1-o(1))mp_{in}$ edge-disjoint
    perfect matchings $\{\mathcal M_{j,k}\}_{k=1}^{L}$.
    For each $k\in [L]$, combining the edges
    in the matchings
    $ \{{\mathcal M}_{j,k}\}_{j \in [\ell-1]}$ gives $m$
    directed paths, each directed from $V_1$ to $V_{\ell}$ and
    covering $\bigcup _{i=1}^{\ell} V_i$. Let $P_{k,1},\ldots ,
    P_{k,m}$ denote these paths and ${\mathcal P}_k =
    \{P_{k,1},\ldots , P_{k,m}\}$.

    Now for each exterior edge $e$ of $D_n({\mathcal V})$ choose a value
    $h(e) \in [L]$ uniformly at random, all values chosen
    independently. Now expose the exterior edges of $F$ and for
    each $i\in [L]$ let $H_i$ denote the subgraph of $F$
    with edge set
    $\{e \in E(F) : e \mbox{ exterior with } h(e) = i \}$.

    \begin{claim} For any $k\in [L]$,
    the digraph $C_k:= \{\overrightarrow{e}: \overrightarrow{e} \mbox { is a directed edge of some path in }{\mathcal P}_k\} \cup H_k$ contains a directed Hamilton     cycle with probability $1 - o(1/n)$.
    \end{claim}
    \noindent Note that the proof of the lemma follows from the claim,
    by taking a union bound over all $k\in [L]$.

    To prove the claim, let
    $${\mathcal M}_k = \{(u_{k,i},v_{k,i}) \in V_1\times V_{\ell }: P_{k,i}
    \mbox{ is a directed }u_{k,i}-v_{k,i} \mbox{ path}\}.$$ By Remark \ref{remark: Ham cycles correspond to Ham cycles} it
    suffices to prove that the auxiliary digraph $C_k({\mathcal M}_k,V_0)$ contains a
    directed Hamilton cycle.
    Note that $|V(C_k({\mathcal M}_k,V_0))| = s + m$ and that
    $C_k({\mathcal M}_k,V_0)\sim {\mathcal D}(s+m, p_{ex}/L)$.
    Since
    $p_{ex}/L =\omega( \log n/(s+m))$, the digraph
    $C_k({\mathcal M}_k,V_0)$ is Hamiltonian with probability $1-o(1/n)$.
    By Remark \ref{remark: Ham cycles correspond to Ham cycles},
    this completes the proof of the claim, and therefore the lemma.
\end{proof}

The following lemma allows us to cover the edges of the complete digraph in a reasonably balanced way using copies of $D_n(\V )$.

\begin{lemma}
    \label{lemma: even partition of the edges}
    Suppose that $X$ is a set of size $n$ and that $\ell , m,s \in {\mathbb N}$ satisfying $n = m\ell +s$, with $t=\omega(\ell \log n)$, $t=\omega((n^2/s^2) \log n)$ and $s= o(n)$ and $m = o(s)$. Let   $\mathcal V^{(1)},\ldots ,\mathcal V^{(t)}$ be a collection of  $(\ell ,s)$-partitions
    of $X$ chosen uniformly and independently at random, where
    ${\mathcal V}^{(i)} = (V^{(i)}_0,\ldots ,V^{(i)}_{\ell})$.
    Then \whp for each pair $u,v\in X$, the directed edge
    $e = \overrightarrow{uv}$ satisfies:
    \begin{enumerate}
        \item $|A_e| = (1 + o(1))\frac {t}{\ell}$ where
        $A_e := \big \{i\in [t]:
        e \mbox{ is an interior edge of }D_n({\mathcal V}^{(i)}) \big \}$.
        \item $|B_e| = (1 +o(1)) \frac {s^2t}{n^2}$ where
        $B_e := \big \{i\in [t]: e \mbox{ is an exterior edge of }
        D_n({\mathcal V}^{(i)}) \big \}$.
    \end{enumerate}
\end{lemma}

\begin{proof}
Let $\mathcal V^{(1)},\ldots ,\mathcal V^{(t)}$ be $(\ell,s)$-partitions chosen uniformly and independently at random. Given a fixed directed edge $e$, the sizes $|A_e|$ and $|B_e|$ are binomially distributed on a set $t$ with
$${\mathbb E}(|A_e|) = (\ell - 1)\frac{m^2}{n(n-1)}t \quad \mbox {and} \quad {\mathbb E}(|B_e|) = \frac{m^2 + 2sm + s(s-1)}{n(n-1)}t.$$
Using that $s = o(n)$ and $n = m\ell +s$ this gives that ${\mathbb E}(|A_e|) = (1+o(1))\frac {t}{\ell }$ and
    using $m = o(s)$ gives ${\mathbb E}(|B_e|) = (1+o(1))\frac {s^2t}{n^2}$. Therefore by Lemma \ref{Chernoff}
    \begin{equation}
        \label{equation: number of appearance of u,v in parts}
        \Pr\Big ( \big ||A_e| - \mathbb{E}(|A_e|) \big | > a \mathbb{E}(|A_e|) \Big )
        \leq 2e^{-a^2\mathbb{E}(|A_e|)/3}
        \leq 2e^{-(1+o(1))a^2t/3\ell }
         =o( 1/n^2).
    \end{equation}
    Here we used that $a^2t/3\ell \geq 3\log n$ for $a=o(1)$.
    Similarly using that $ts^2/n^2= \omega (\log n)$ we find
    $\Pr\Big ( \big ||B_e| - \mathbb{E}(|B_e|) \big |
    > a \mathbb{E}(|B_e|) \Big ) = o( 1/n^2)$. Taking a union bound
    over all directed edges gives that \whp 1. and 2. hold for all
    $e$.    \end{proof}

Now we are ready to prove Theorem \ref{thm:PackingRandom}.

\begin{proof}[Proof of Theorem \ref{thm:PackingRandom}]
    Let $\alpha = \alpha (n)$ be some function tending to infinity
    arbitrarily slowly with $n$. Suppose that $p \geq \alpha ^6 \log ^4 n/ n$ and
    let $\ell = \alpha ^3 \log n$, $s = n/\alpha ^2\log n$
    be integers with $n=m\ell + s$. Note that this gives
    $m = (1+o(1))n/\alpha ^3\log n$. Additionally set $t = \alpha ^5 \log ^3 n$.
    With these choices, the hypothesis of Lemma
    \ref{lemma: even partition of the edges} is satisfied.
    Let ${\mathcal V}^{(1)},\ldots ,{\mathcal V}^{(t)}$ be a
    collection of $(\ell ,s)$-partitions of $X = [n]$, chosen
    so that the conclusions of Lemma
    \ref{lemma: even partition of the edges} are satisfied.
    Therefore $|A_e| = (1 + o(1))t/\ell = (1+o(1))\alpha ^2\log ^2n$ and
    $|B_e| = (1+ o(1))s^2t/n^2 = (1+o(1))\alpha \log n$ for every $e$.

    To begin, whenever we expose the edges of a
    directed graph $D \sim  {\cal D}(n,p)$, we will
    assign the edges of $D$ among $t$ edge
    disjoint subdigraphs $D^{(1)},\ldots ,
    D^{(t)}$. The digraphs
    $D^{(i)}$ are constructed as follows.
    For each edge $e$ independently choose a random value
    $h(e) \in A_e \cup B_e$
    where an element in $A_e$ is selected with probability
    $(1-1/\alpha)/|A_e|$
    and an element in $B_e$ is selected with probability $1/\alpha |B_e|$.
    For each $i\in [t]$, we take $D^{(i)} $ to be the
    digraph given by $D^{(i)} = \{e \in E(D): h(e) = i\}$. We prove
    that \whp $D^{(i)}$ contains edge-disjoint
    Hamilton cycles covering almost all of its edges.

    First note that all edges $e$ of
    $D_n({\mathcal V}^{(i)})$ appear independently in
    $D^{(i)}$. If $e$ is an interior edge then the
    probability that it appears is
    $p(1-1/\alpha )/|A_e| \geq (1- o(1))p/\alpha ^2 \log ^2n := p_{in}$.
    Similarly, each exterior edge $e$ in
    $D_n({\mathcal V}^{(i)})$ appears in $D^{(i)}$
    with probability
    $p/\alpha |B_e| \geq (1-o(1))p/\alpha ^2 \log n := p_{ex}$.
    Using these values, select $F$ as in Lemma
    \ref{lemma: building cycles in (l,s)-partitions}.
    Also set $L = (1-o(1))mp_{in}$.
    Due to monotonicity we conclude that for every $i \in [t]$ we have
    \begin{align}
        \label{equation: coupling of probabilities}
        \Pr(D^{(i)} \mbox{ contains } L \mbox{ edge disjoint Ham. cycles})
        		 \geq
        \Pr(F \mbox{ contains } L \mbox{ edge disjoint Ham. cycles}).
    \end{align}
	Now we claim with these choices of $p_{in}$ and $p_{ex}$ the hypothesis
	of Lemma \ref{lemma: building cycles in (l,s)-partitions}
	are satisfied. Indeed, using $p\geq \alpha ^6\log ^4n/n$ gives
	$$(1+o(1))p_{in} = \frac{p}{\alpha ^2 \log ^2n} \geq
	\frac{\alpha ^4 \log ^2n}{n} = (1+o(1))\frac{\alpha  \log n}{m},$$
	and so $p_{in} = \omega (\log n/m)$. Similarly we have
	\begin{align*}
	p_{ex} &= (1+o(1))\frac{p}{\alpha ^2 \log n} =
	(1+o(1))p_{in } \log n
	= (1+o(1)) \frac{\alpha m p_{in}\log n}{ s},
	\end{align*}
	and $p_{ex} = \omega (mp_{in}\log n/(m+s))$. Thus by
	Lemma \ref{lemma: building cycles in (l,s)-partitions},
    $\Pr(F \mbox{ contains }
    L \mbox{ edge disjoint Ham. cycles})
    = 1 - o(1)$.
    Summing over $i\in [t]$ and combining with
    \eqref{equation: coupling of probabilities},
    this proves that \whp $D$ contains at least
    $(1-o(1))Lt = (1-o(1))mp_{in}t
        = (1-o(1)). \frac{n - s}{\ell}. \frac{p\ell}{t}.t
        = (1-o(1))np$ edge-disjoint Hamilton cycles.
\end{proof}

%
%
%
%
%
%
%
%
%
%
%
%
%
%
%
%
%
%
%
%

\section{Covering $\mathcal D(n,p)$ with Hamilton cycles}
In this section we prove Theorem \ref{thm:CoverRandom}. To begin we first prove the following lemma. The proof makes use of the max-flow min-cut theorem
and the integrality theorem for network flows (see Chapter 7 in \cite{vanLintWilson}).

\begin{lemma}
    \label{lemma: completing arbitrary small graph into an r-factor}
    Let $G = (A,B,E)$ be a bipartite graph, with $|A| = |B| = N$ and $\delta (G) \geq d$. Suppose that $G$ has the following properties:
    \begin{itemize}
        \item For any $X \subset A$, $Y \subset B$ with
        $|X| \geq \frac{N}{4}$ and $|Y| \geq \frac{N}{4}$ we
        have $e_G(X,Y) \geq \frac{d N}{40}$,
        \item For any $X \subset A$ with $|X| \leq \frac{N}{4}$,
        if $e_G(X,Y) \geq \frac{3d|X|}{4}$ for some $Y \subset B$
        then $|Y| \geq 2|X|$,
        \item For any $Y \subset B$ with $|Y| \leq \frac{N}{4}$,
        if $e_G(X,Y) \geq \frac{3d|Y|}{4}$ for some $X\subset A$
        then $|X| \geq 2|Y|$.
    \end{itemize}
    Then given any integer $r$ with $r \leq \frac{d}{80}$ and a bipartite graph $H$ on vertex set $A\cup B$ with $\Delta: = \Delta (H) \leq \frac{r}{2}$, there exists a subgraph $G'$ of $G$ which is edge disjoint from $H$ such that $G' \cup H$ is $r$-regular.
\end{lemma}

\begin{proof}
    Given a graph $F$ on $V(G)$ and a vertex $v$ of $G$, let
    $d_F(v)$ denote the degree of $v$ in $F$.
    By assumption, we have $d_H(v) \leq r/2$ for all $v \in A \cup B$.
    We wish to find a subgraph $G'$ of $G$ which is edge-disjoint from
    $H$ so that $d_{G'}(v) + d_H(v) = r$ for all $v \in V(G)$.
    We prove the existence of $G'$ by representing it as a flow in
    an appropriate network.

    Consider the following network $D$ on vertex set
    $V(G) \cup \{s,t\}$, with source $s$ and sink $t$.
    For each $a\in A$, the edge
    $\overrightarrow{sa} \in E(D)$ and it has capacity $r - d_H(a)$.
    For each $b\in B$, the edge $\overrightarrow{bt} \in E(D)$
    and it has capacity $r - d_H(b)$.
    Lastly, each edge in $E(G \setminus H)$ is directed from $A$ to
    $B$ and has capacity $1$. Using the integrality theorem for
    network flows, it is sufficient to show that there is
    a flow from $s$ to $t$ of value
    \begin{equation}
        \label{equation: desired value of flow}
        V = \sum _{a\in A} (r - d_H(a)) = rN - \sum _{a\in A} d_H(a).
    \end{equation}
    By the max-flow min-cut theorem it is sufficient to show that
    $D$ does not contain an $s-t$ cut of capacity less than $V$.

    To see this, suppose for contradiction that
    $\{s\} \cup A_s \cup B_s$ and $A_t \cup B_t \cup \{t\}$
    forms such a cut, $A_v \subset A$ and $B_v \subset B$ for
    $v\in \{s,t\}$. The capacity of this cut is
    \begin{align*}
        C = \sum _{a \in A_t} (r - d_H(a))
        + \sum _{b\in B_s} (r - d_H(b))
        + e_{G\setminus H}(A_s,B_t).
    \end{align*}
    We may assume that $|A_s| \leq N/4$ or $|B_t| \leq N/4$.
    Indeed, otherwise from the statement of the lemma we have
    $e_G(A_s,B_t) \geq dN/40$  and
    \begin{align*}
        C \geq e_{G\setminus H}(A_s,B_t)
        \geq e_{G}(A_s,B_t) - \Delta N \geq
        {dN}/{40} - {r N}/{2} \geq rN \geq V,
    \end{align*}
    since $r \leq d/80$. We will focus on the case $|A_s| \leq N/4$
    as the case $|B_t| \leq N/4$ follows from an identical argument.

    Note that since $e_{G\setminus H}(A_s, B) \geq
    (\delta (G) - \Delta )|A_s| \geq (d - \Delta )|A_s|$, we find
    \begin{equation}
        \label{equation: middle cut}
        e_{G\setminus H}(A_s ,B_t)
        \geq
        e_{G\setminus H}(A_s,B) - e_{G\setminus H}(A_s,B_s)
        \geq
        (d- \Delta )|A_s| - e_{G}(A_s,B_s).
    \end{equation}
    From \eqref{equation: middle cut} it follows that if
    $e_{G}(A_s,B_s) \leq \frac{3d|A_s|}{4}$ then
    \begin{align*}
        C \geq \sum _{a\in A_t} (r - d_H(a)) +
        e_{G\setminus H}(A_s ,B_t)
        & \geq
        \sum _{a\in A_t} (r - d_H(a)) + (d - \Delta - \frac{3d}{4}) |A_s|\\
        & \geq
        \sum _{a\in A_t} (r - d_H(a)) + r |A_s| \geq V,
    \end{align*}
    where the second last inequality holds since $d/4 \geq 2r \geq \Delta +r$ and
    the last inequality holds by \eqref{equation: desired value of flow}.
    If $e_{G}(A_s,B_s) \geq \frac{3d|A_s|}{4}$, since $|A_s| \leq
    |A|/4$, by the hypothesis of the lemma we have
    $|B_s| \geq 2|A_s|$. But then, since $\Delta \leq r/2$ we have
    \begin{align*}
        C
        \geq
        \sum _{a \in A_t} (r - d_H(a)) + \sum _{b \in B_s} (r - d_H(b))
        & \geq
        \sum _{a \in A_t} (r - d_H(a)) + |B_s|(r - \Delta) \\
        & \geq
        \sum _{a \in A_t} (r - d_H(a)) + 2|A_s|\times \frac{r}{2} \\
        & \geq \sum _{a\in A} (r - d_H(a)) = V.
    \end{align*}
    This covers all cases, and completes the proof.
\end{proof}

We now prove a covering version of Lemma \ref{lemma: building cycles in (l,s)-partitions}. In our proof of Theorem \ref{thm:CoverRandom} we will again break $D \sim {\cal D}(n,p)$ into many sub-digraphs which are distributed similarly to $F$ from Lemma \ref{lemma: building cycles in (l,s)-partitions}. However there will be some small fluctuation in the edge probabilities of edges in these sub-digraphs. The slightly unusual phrasing of the next lemma is intended to allow for these fluctuations.

\begin{lemma}
    \label{lemma: covering edges with cycles in (l,s)-partitions}
    Let ${\mathcal V} = (V_0, V_1,\ldots,V_{\ell})$ be an
    $(\ell ,s)$-partition of a set $X$ of size $n = m\ell  +s$.
    Choose a random subdigraph $F$ of
    $D_n({\mathcal V})$ as follows:
    	\begin{itemize}
    		\item include each interior edge $e$ from $D_n({\cal V})$
    		independently with  probability  
    		 $q_e\in (1\pm o(1))p_{in}$;
    		\item include each exterior edge from $D_n({\cal V})$
    		independently with probability at least $p_{ex}$.
    	\end{itemize}
    Then, provided $p_{in}= \omega( \log n/m)$ and
    $p_{ex} =\omega(\log n/(m+s))$, with probability
    $1-o(1/n^2)$ there are $(1+o(1))mp_{in}$ directed Hamilton cycles in
    $F$ which cover all interior edges of $F \cap D_n({\cal V})$.
\end{lemma}

\begin{proof}
    We begin by exposing the interior edges of $F$.
    For any $j\in [\ell-1]$, all of
    edges $E_{D_n}(V_j,V_{j+1})$ appear in $E_F(V_j,V_{j+1})$
    independently with probability at least  $(1-o(1))p_{in}$.
    For any $j\in [\ell-1]$, let $F_j$ be the subdigraph of $F$ consists of the vertices $V_j\cup V_{j+1}$ and the edges in $E_F(V_j,V_{j+1}))$.
    We again view $F_j$ as a bipartite graph, simply by ignoring the orientations.
    As in Lemma \ref{lemma: building cycles in (l,s)-partitions},
    with probability $1-o(1/n^2)$ for each $j\in [\ell -1]$ we can find
    $L = (1-o(1))mp_{in}$ edge-disjoint perfect
    matchings in $E_F(V_j,V_{j+1})$, which we denote by
    $\{{\mathcal M}_{j,k}\}_{k=1}^{L}$. Now remove the edges of these
    matchings from $E_F(V_j,V_{j+1})$ and let $H_j$
    denote the remaining subdigraph. Since $p_{in} = \omega (\log n/m)$ and $q\in (1+o(1))p_{in}$,
    by Chernoff's inequality, with probability $1-o(1/n^2)$
    every vertex $u \in V_j$ and $v\in V_{j+1}$
    satisfies
    	$$ (1+o(1))mp_{in}
    		\leq
    	d^+_{F_j}(u), d^-_{F_j}(v)
    		\leq
    	(1+o(1))mp_{in}.$$
    Therefore with probability $1-o(1/n^2)$, for all $j\in [\ell -1]$,
    such $u$ and $v$ satisfy
    	\begin{align}
    		\label{equation: H degree control}
    		d^+_{H_j}(u) = o(mp_{in})
    				\quad
    		\mbox{ and }
    				\quad
    		d^-_{H_j}(v) = o(mp_{in}).
    	\end{align}

    Now given $X \subset V_j$ and $Y \subset V_{j+1}$ we also have
    ${\mathbb E} \big (e_{F_j}(X,Y)\big ) = (1 \pm o(1))|X||Y|p_{in}$.
    Chernoff's inequality therefore shows that
    	\begin{equation*}
    		Pr \big ( \big |e_{F_j}(X,Y) - (1 \pm o(1))|X||Y|p_{in}
    		\big | > t \big ) \leq e^{-t^2 /4|X||Y|p_{in}}.
    	\end{equation*}
    Using this bound it is easy to check that the following holds:
    with probability
    $1 - n^{\omega (1)}$, for all $j\in [\ell -1]$ the hypothesis of Lemma
    \ref{lemma: completing arbitrary small graph into an r-factor}
	are satisfied by the bipartite graph $F_j$,
	taking $d = (1- o(1))mp_{in}$ and $N = m$.
    Setting $r = \max _{j\in [\ell -1]} \{ 2\Delta (H_j)\}$, from
    \eqref{equation: H degree control} we have
    $r \ll d$ for all $j \in [\ell -1]$. Therefore by
    Lemma \ref{lemma: completing arbitrary small graph into an r-factor},
    with probability $1- n^{-\omega (1)}$,
    for all $j\in [\ell -1]$ the graph $F_j$
    contains an $r$-regular subgraph $G_j$ which includes
    all edges of $H_j$.

    Now by Hall's theorem, for each $j\in [\ell -1]$ the digraph
    $G_j$ can be decomposed into $r$ edge-disjoint perfect
    matchings, which we denote by $\{M_{j,k}\}_{k=L+1}^{L+r}$.
	Combined with the matchings at the beginning of the proof, we have
	show that with probability $1 - o(1/n^2)$, for each $j\in [\ell -1]$
	there are perfect matchings $\{{\cal M}_{j,k}\}_{k=1}^{L + r}$
	which cover all interior directed edges of $F$. By combining
	the edges $\{{\mathcal M}_{j,k}\}_{j=1}^{\ell-1}$ for each
	$k\in [L+r]$, we get $m$ directed paths, each directed from
	$V_1$ to $V_{\ell}$ and covering
	$\bigcup _{i=1}^{\ell} V_i$. Let $P_{k,1},\ldots ,
    P_{k,m}$ denote these paths and ${\mathcal P}_k =
    \{P_{k,1},\ldots , P_{k,m}\}$. In particular these paths
    cover all interior edges of $F$.

    Now to complete the proof we expose the exterior edges of $F$
    and use them to complete each ${\mathcal P}_k$ into a directed
    Hamilton cycle as in the proof of Lemma
    \ref{lemma: building cycles in (l,s)-partitions}. For each
    $k\in [L + r]$ let
    $${\mathcal M}_k = \{(u_{k,i},v_{k,i}) \in V_1\times V_{\ell }: P_{k,i}
    \mbox{ is a directed }u_{k,i}-v_{k,i} \mbox{ path}\}.$$
    Now $|V(F({\mathcal M}_k,V_0))| = s + m$ and
    as in Lemma \ref{lemma: building cycles in (l,s)-partitions},
    $F({\mathcal M}_k,V_0)\sim {\mathcal D}(s+m, p_{ex})$.
    Since $p_{ex} =\omega( \log n/(s+m))$, the digraph
    $F({\mathcal M}_k,V_0)$ is Hamiltonian with probability $1-o(1/n^3)$.
    By Remark \ref{remark: Ham cycles correspond to Ham cycles}, this
    shows that with probability $1 - o(1/n^2)$,
    for all $k\in [L + r]$ the digraph $F$ contains a directed Hamilton cycle
    containing all edges of the paths in ${\cal P}_k$. As these paths
    cover all interior edges of $F$, this completes the proof of the lemma.
\end{proof}

\begin{proof}[Proof of Theorem \ref{thm:CoverRandom}]
    The proof follows a similar argument to that of Theorem
    \ref{thm:PackingRandom}.
    Let $\alpha = \alpha (n)$ be some function tending arbitrarily slowly to
    infinity with $n$ and let $p \geq \alpha ^4 \log ^2n/n$.
    Let $n=m\ell + s$ where
    $\ell = \alpha $, $s = n/\alpha $ and $t = \alpha ^2 \log  n$. Note
    that $m = (1+o(1))n/\alpha $.

    Since $t =\omega (\ell \log n)$ we can take
    ${\mathcal V}^{(1)},\ldots ,{\mathcal V}^{(t)}$ to be a
    collection of $(\ell ,s)$-partitions of $X = [n]$
    as
    given by Lemma \ref{lemma: even partition of the edges}.
    To begin, whenever we expose the edges of a directed graph
    $D \sim {\mathcal D}(n,p)$, we will
    assign the edges among $t$ sub-digraphs
    $D^{(1)},\ldots ,
    D^{(t)}$. The digraphs
    $D^{(i)}$ are constructed as follows.
    Let
    \begin{equation*}
        A_e := \{i\in [t]: e \mbox{ is interior in }D_n({\mathcal V}^{(i)})\}.
    \end{equation*}
    By Lemma \ref{lemma: even partition of the edges}, \whp
    for each edge $e$ we have
    $|A_e| = (1+o(1)) t/\ell = (1+o(1) \alpha \log n$.
    Independently for each edge $e$ choose a value $h(e) \in A_e$
    uniformly at random. For each $i\in [t]$, let the digraph
    $D^{(i)} $ contain the edges $\{e\in E(D): h(e) = i\}$. Furthermore,
    adjoin all edges of $D$ which occur as an exterior
    edge of $D_n({\mathcal V}^{(i)})$ to $D^{(i)}$. We will prove that
    w.h.p. $D^{(i)}$ contains directed Hamilton cycles
    covering all the edges of $D \cap D_n({\cal V}^{(i)})$.

    First note that all edges $e$ of
    $D_n({\mathcal V}^{(i)})$ appear independently in
    $D^{(i)}$. If $e$ is an interior edge then the
    probability that it appears is
    $p/|A_e| = (1\pm o(1))p/\alpha \log n = p_{in}$. We see that each
    interior edge
    of $D_n({\mathcal V}^{(i)})$ appears in $D^{(i)}$ independently
    with probability between $(1-o(1))p_{in}$ and $(1+o(1))p_{in}$.
    Also, each exterior edge $e$ in $D_n({\mathcal V}^{(i)})$ appears in
    $D^{(i)}$ with probability $p_{ex}: = p$. Now we have
    $p_{ex} = p =\omega (\frac {\log n}{m+s})$. We also
    have $p_{in} = (1 + o(1))\frac{p\ell }{t} \geq \frac {\alpha ^2 \log n}{n}
    = \frac {\alpha \log n}{m}$,
    so $p_{in} =\omega (\frac {\log n}{m})$. Thus by Lemma
    \ref{lemma: covering edges with cycles in (l,s)-partitions} we obtain
    \begin{align}
        \label{equation: individual coverings with Hamilton cycles}
        \Pr(D^{(i)} \mbox{ has } (1+o(1))mp_{in}
        \mbox{ directed Hamilton cycles covering its interior edges}) \geq
        1 - \frac{1}{n^2}.
    \end{align}
    Summing \eqref{equation: individual coverings with Hamilton cycles}
    over $i\in [t]$, this proves that w.h.p. $D$ contains
    $(1+o(1))mp_{in}t = (1+o(1))np$  Hamilton cycles covering the
    interior edges of $D^{(i)}$ for all $i\in [t]$. Since each edge
    of $D$ occurs as an interior edge of $D^{(i)}$ for some $i \in [t]$, this
    completes the proof of the theorem.
\end{proof}

%
%
%
%
%
%
%
%
%
%
%
%
%
%
%
%
%
%
%
%

\section{Packing Hamilton cycles in pseudo-random directed graphs}

\subsection{Pseudo-random digraphs and Hamiltonicity}\label{sec:HamInPseudo}

\begin{definition} \label{definition:pseudorandom} A directed graph $D$ on $n$ vertices is called $(n,\lambda,p)$-\emph{pseudo-random} if the following hold:
    \begin{description}
        \item[(P1)] $(1 -\lambda)np \leq
        d_D^{+}(v), d_D^{-}(v)\leq (1 +\lambda)np$ for every $v\in V(D)$;
        \item[(P2)] For every $X\subseteq V(D)$ of size $|X|\leq\frac{4\log^8n}{p}$ we have
        $e_D(X)\leq (1-\lambda)|X|\log^{8.02}n$;
        \item[(P3)] For every two disjoint subsets $X,Y\subseteq V(D)$ of sizes
        $ |X|,|Y|\geq \frac{\log^{1.1}n}p$ we have $e_D(X,Y)=
        (1\pm\lambda)|X||Y|p$.
    \end{description}
\end{definition}

The following theorem of Ferber, Nenadov, Noever, Peter
and Skori\'c \cite{ferberrobust} gives a sufficient
condition for pseudo-random digraph to be Hamiltonian.

\begin{theorem}(Theorem 3.2, \cite{ferberrobust})\label{thm:pseudorandom is Hamiltonian}
    Let $0<\lambda<1/10$. Then for $p = \omega(\frac{\log^8
        n}{n})$ the following holds. Let $D$ be a directed graph with
        the following properties:
    \begin{description}
    \item [(P1)] $(1 -\lambda)np \leq
        d_D^{+}(v), d_D^{-}(v)\leq (1 +\lambda)np$ for every $v\in V(D)$;
    \item [(P2)*] for every $X\subseteq V(D)$ of size $|X|\leq \frac{\log^2n}{p}$ we have
    $e_D(X)\leq |X|\log^{2.1}n$;
    \item[(P3)*] for every two disjoint subsets $X,Y\subseteq V(D)$ of sizes
    $ |X|,|Y|\geq \frac{\log^{1.1}n}p$ we have $e_D(X,Y)\leq
    (1+\lambda)|X||Y|p$.
    \end{description}
    Then $D$ contains a Hamilton cycle.
\end{theorem}

\subsection{Properties of pseudo-random graphs}\label{sec:PropPseudo}

The following lemmas will be useful in the proof of Theorem \ref{thm:PackingPseudoRandom}. We have deferred the proofs to the Appendix. In
these lemmas we assume that $p=\omega(\log^{14}n/n)$, $p'=p/\log^6n$,
$s=\sqrt{n/\alpha p'}$ and $m=s/ \log n$.

\begin{lemma}
\label{lemma:auxiliary2 is pseudorandom}
    Let $D$ be a $(n,\lambda,p)$-pseudo-random digraph
     with $0<\lambda<1$ and $p=\omega(\log^{14}n/n)$.
     We first select a random subdigraph
     $C$ of $D$ by including edges independently with probability
     $q\in(1 \pm o(1)) p'/p$. Then select an $(\ell,s)$-partition of $V(D )$
     given by ${\mathcal V}=(V_0, V_1, \dots , V_\ell)$
     uniformly at random, with  $|V_0|=s$ and $n=m\ell +s$.
    Then with probability
    $1 - o(1/n)$ the following holds:
    for every collection $\M$ of $m$ disjoint pairs from $V_1\times V_\ell$,
    the random digraph $F_0= C({\mathcal M},V_0)$
    satisfies the following properties:
    \begin{enumerate}[$(A)$]
        \item
        $(1 - 3\lambda)(s+m) p' \leq d^{+}_{F_0}(v), d^{-}_{F_0}(v) \leq  (1 + 3\lambda)(s+m) p'$ for every $v \in V(F_0)$,
        \item we have
        $e_{F_0}(X)\leq |X|\log^{2.1}n$ for every $X\subseteq V(F_0)$ of size $|X|\leq  \frac{\log^2 (s+m)}{p'}$,
        \item for every two disjoint subsets $X,Y\subseteq V(F_0)$ of sizes
        $|X|,|Y|\geq \frac{\log^{1.1} (s+m)}{p'}$, we have
        \label{property:edges-X-Y}
        \begin{align*}
            &e_{F_0}(X,Y)\leq(1 + 2\lambda)|X||Y|p'.
        \end{align*}
    \end{enumerate}

\end{lemma}

\begin{lemma}\label{lemma:bipartite is pseudorandom}
    Let $D$ be $(n,\lambda,p)$ pseudo-random digraph with $0<\lambda<1/4$ and $p=\omega(\log^{14}n/n)$.  Suppose that $V(D)=V_0\cup V_1\cup\dots\cup V_\ell$ is a random $(\ell,s)$-partition of $V(D)$ with
    $|V_0|=s$ and $n=m\ell +s$ and let $F$ be the graph obtained from $D$ by keeping every interior edge with probability $p_{in}= 1/(\alpha\ell\log n) $. Then with probability $1-o(1/n)$, for every $j\in [\ell-1]$ the directed subgraph $F_j=E_F(V_j,V_{j+1})$
    contains
    $(1-4\lambda)mp\cdot p_{in}$ edge disjoint perfect matchings.
\end{lemma}

\subsection{Proof of Theorem \ref{thm:PackingPseudoRandom}}

To prove Theorem \ref{thm:PackingPseudoRandom} we use the following lemma (the analogue of Lemma \ref{lemma: building cycles in (l,s)-partitions}) about the existence of many edge disjoint Hamilton cycles in special pseudo-random directed graphs.

\begin{lemma}
    \label{lemma: building cycles in (l,s)-partitions for pseudorandom}
    Let ${\mathcal V} = (V_0, V_1,\ldots,V_{\ell})$ be an  $(\ell ,s)$-partition
    of a set $X$ of size $n=\ell m+s$,
    chosen uniformly and
    independently at random.
    Let $D$ be an $(n,\lambda,p)$ pseudo-random graph on the vertex set $X$,
    with $p=\omega(\log^{14}n/n)$, and $0<\lambda <1/100$. Suppose that we select  a random subdigraph $F$ of $D(\V)$ as follows:
    \begin{itemize}
    	\item include each interior edge of $D(\V)$ independently with probability $p_{in}$;
    	\item include each exterior edge of $D(\V)$ independently with probability $p_{ex}$.
    \end{itemize}
    Then, provided $p_{in}= (1-o(1))\cdot 1/(\alpha\ell \log n)$ and $p_{ex} =n^2/(\alpha^2 s^2\ell ^2\log n)$, where $p'=p/\log^6n$, $s=\sqrt {n/\alpha p'}$, $m=s/ \log n$ and $\alpha = \alpha (n)$ is some function tending arbitrarily slowly to infinity with $n$, $F$ contains $(1-o(1))(1-4\lambda)mpp_{in}$
    edge-disjoint Hamilton cycles with probability $1-o(1/n)$.
%
\end{lemma}

\begin{proof}
    To begin, look at the interior edges of $F$.
    For $j\in [\ell-1]$ all
    edges of $E_D(V_j,V_{j+1})$ appear in $F$ independently with probability $p_{in}$. Lemma \ref{lemma:bipartite is pseudorandom} therefore gives that
     with probability $1-o(1/t)$,
    $E_F(V_j,V_{j+1})$ contains
    $L:=(1-4\lambda)mpp_{in}$ edge-disjoint perfect matchings
    $\{\mathcal M_{j,k}\}_{k=1}^{L}$ for all $j\in [\ell-1]$.
    For each $k\in [L]$, taking the union of the edges
    in the matchings
    $\bigcup _{j=1}^{\ell-1} {\mathcal M}_{j,k}$ gives $m$
    directed paths, each directed from $V_1$ to $V_{\ell}$ and
    covering $\bigcup _{i=1}^{\ell} V_i$. Let $P_{k,1},\ldots ,
    P_{k,m}$ denote these paths and ${\mathcal P}_k =
    \{P_{k,1},\ldots , P_{k,m}\}$.

    Now assign to each exterior edge $e$ of $D({\mathcal V})$ a value
    $h(e) \in [L]$ chosen uniformly at random, all values  chosen
    independently. Look at the exterior edges of $F$ and for
    each $i\in [L]$ let $H_i$ denote the subgraph of $F$
    with edge set
    $\{e \in E(F) : e \mbox{ exterior with } h(e) = i \}$.

    \begin{claim} For any $k\in [L]$, the digraph
    $C_k:=\{\vec{e}\ :\ \vec{e}\ is\ a\ directed\ edge\ of\ some\ path\ is\ \mathcal P_k\} \cup H _k$ contains a directed Hamilton
    cycle with probability $1 - o(1/n)$.
    \end{claim}
    Note the proof of the lemma immediately follows from the claim,
    summing over $k\in [L]$.

    To prove the claim, let
    $${\mathcal M}_k = \{(u_{k,i},v_{k,i}) \in V_1\times V_{\ell }: P_{k,i}
    \mbox{ is a }u_{k,i}-v_{k,i} \mbox{ directed path}\}.$$

 By Remark \ref{remark: Ham cycles correspond to Ham cycles} it
    suffices to prove that the auxiliary digraph ${C}_k({\mathcal M}_k,V_0)$ contains a
    directed Hamilton cycle.
    Now note that $|V({C}_k(\M_k,V_0))| = s + m$. We now wish to prove that with probability $1-o(1/n)$ every $C_k(\M_k,V_0)$ is Hamiltonian.
    Observe that each $C_k(\M_k,V_0)$ was created from $F(\M_k,V_0)$ by keeping each edge $e$ with probability at least
        \begin{equation*}
            p_{ex}/|L| = \frac{n^2}{\alpha ^2 s^2 \ell ^2 \log n} \times
            \frac {1}{(1-o(1))mp p_{in} }.
        \end{equation*}
    Using that $p_{in} = (1-o(1))1/(\alpha \ell \log n)$,  $p ' = n/\alpha s^2$ and that $m \ell = (1-o(1))n$ gives that $p_{ex}/|L| = (1-o(1))p'/p$.

        By applying Lemma \ref{lemma:auxiliary2 is pseudorandom}, we see
        that ${ C}_k({\mathcal M}_k,V_0)$ satisfies properties $(A)$, $(B)$ and
        $(C)$ with probability $1- o(1/n)$. But (A), (B) and (C)
        give properties $(P1)$, $P(2)^*$ and $P(3)^*$ from
        Theorem \ref{thm:pseudorandom is Hamiltonian}, taking $p'$ in place
        of $p$. Since $p' = (1-o(1))p/\log ^6n = \omega (\log ^8n/n)$, by
        Theorem \ref{thm:pseudorandom is Hamiltonian} any such
        ${C}_{k}({\mathcal M}_k,V_0)$ are Hamiltonian. But if
        ${C}_k({\mathcal M}_k,V_0)$ is Hamitonian, then so
        is ${ C}_k$. Thus, this proves that ${ C}_k$ is
        Hamiltonian with probability $1 - o(1/n)$.
\end{proof}

Now we are ready to prove Theorem \ref{thm:PackingPseudoRandom}.

\begin{proof}[Proof of Theorem \ref{thm:PackingPseudoRandom}]
    Let $\alpha = \alpha (n)$ be some function tending arbitrarily slowly to infinity with $n$.
    Let $n=m\ell + s$ where
    $m=s/\log n$ and $s = \sqrt{n/\alpha p'}$.

    Let ${\mathcal V}^{(1)},\ldots ,{\mathcal V}^{(t)}$ be a collection of  $(\ell ,s)$-partitions
    of $X=[n]$, where ${\mathcal V}^{(i)} = (V^{(i)}_0,\ldots ,V^{(i)}_{\ell})$,
    chosen uniformly and
    independently at random and $t = \alpha\ell^2\log n$.
    We will
    assign the edges of ${D}$ among $t$ edge
    disjoint subdigraphs ${ D}^{(1)},\ldots ,
    { D}^{(t)}$ such that each $D^{(i)}$ preserves some pseudo-random properties. The digraphs
    ${ D}^{(i)}$ are constructed as follows.
    Let
    \begin{equation*}
        A_e := \{i\in [t]: e \mbox{ is interior in }D({\mathcal V}^{(i)})\};
        \qquad
        B_e := \{i\in [t]: e \mbox{ is exterior in }D({\mathcal V}^{(i)})\}.
    \end{equation*}
    By Lemma \ref{lemma: even partition of the edges}, \whp for each edge $e$ we
    have $|A_e| = (1+o(1)) \frac {t}{\ell}$ and $|B_e| = (
    (1 + o(1)) \frac {s^2t}{n^2}$. We will now show that there exists a function $f$, $f(e) \in A_e \cup B_e$,  such that if ${ D}^{(i)} $ is the
    digraph
    given by ${ D}^{(i)} = \{e: f(e) = i\}$, then $D^{(i)}$ contains $L:=(1-o(1))np/t$ directed Hamilton cycles. Clearly, this will complete the proof.

    For each edge $e$ choose a random value $f(e) \in A_e \cup B_e$
    where each element in $A_e$ is selected with probability
    $(1-1/\alpha)/|A_e|$
    and each element in $B_e$ is selected with probability $1/\alpha |B_e|$.
    For each $i\in [t]$, we take ${ D}^{(i)} $ to be the
    digraph
    given by ${ D}^{(i)} = \{e: f(e) = i\}$. First note that all edges $e$ of
    ${D}({\mathcal V}^{(i)})$ appear independently in
    ${D}^{(i)}$. If $e\in E(D)$ is an interior edge then the
    probability that it appears is
    $(1-1/\alpha )/|A_e| \geq (1- o(1))\ell/t = (1-o(1))1/\alpha \ell \log n:= p_{in}$,
    since $t = \alpha \ell ^2 \log n$.
    Similarly, each exterior edge $e$ in
    $D\cap D({\mathcal V}^{(i)})$ appears in ${\mathcal D}^{(i)}$
    with probability
    $1/\alpha |B_e| \geq (1-o(1))n^2/\alpha ts^2 =
    (1-o(1)) n^2/\alpha ^2 \ell ^2 s^2 \log n   := p_{ex}$.

    Now note that the conditions of Lemma
    \ref{lemma: building cycles in (l,s)-partitions for pseudorandom} are satisfied with
    these values (with $D^{(i)}$ in place of $F$), so  with probability $1-o(1/n)$,  ${D}^{(i)}$ contains
    $L = (1-o(1))(1-4\lambda)mpp_{in}$  edge disjoint Hamilton cycles. Therefore 	with probability $1 - o(1)$, ${D}^{(i)}$ contains
    $L$ edge disjoint Hamilton cycles for each $i\in [t]$. Fix a choice
    of ${\mathcal V}^{(1)},
    \ldots ,{\mathcal V}^{(t)}$ and $f$ such that this holds. Using
    that $p_{in}t = \ell$ this gives that
    ${D}$ contains $(1-o(1))Lt =
    (1-4\lambda - o(1))mpp_{in}t = (1-4\lambda - o(1)) mp \ell \geq (1-5\lambda)np$ edge-disjoint Hamilton cycles, as required.
\end{proof}

 \noindent {\bf Acknowledgment.} The authors would like to thank the referee of the paper for his careful reading and many helpful remarks.

%
%
%
%
%
%
%
%
%
%
%
%
%
%
%
%
%
%
%
%

    \bibliographystyle{abbrv}

    \bibliography{PackCounAndCove}

\newpage

%
%
%
%
%
%
%
%
%
%
%
%
%
%
%
%
%
%
%
%

\section*{Appendix}

\begin{proof}[Proof of Lemma \ref{lemma:auxiliary2 is pseudorandom}.]
    Recall that $p=\omega(\log^{14}n/n)$ and $p'=p/\log^6n$,
    $s=\sqrt {n/\alpha p'}$, $m=s/\log n$, where $\alpha = \alpha (n)$ is
    some function tending arbitrarily slowly to infinity with $n$.

    We will first prove (A). Note that  for
    any choice of ${\mathcal M}$, and any vertex $v\in V(F_0)$ we have
    $d^+_{F_0}(v) = |N^+_{F_0}(u) \cap V_1| + |N^+_{F_0}(u) \cap V_0|$
    for some $u\in V(D)$. Similarly $d^-_{F_0}(v) =
    |N^-_{F_0}(w) \cap V_{\ell }| + |N^-_{F_0}(w) \cap V_0|$
    for some $w\in V(D)$. Let us thus estimate  $|N^-_{F_0}(v) \cap V_{i}|$ for $i\in \{0,1,\ell\}$ and for every $v\in V(D)$.
    Recall that  $|N^{\pm}_D(v)| = (1\pm \lambda )np$ by Definition \ref{definition:pseudorandom} \textbf{(P1)}.
    As edges remain independently with probability $(1-o(1))p'/p$ and  ${\mathcal V}$ is chosen uniformly at random,
    for every vertex $v\in V(D)$ we have that
    $${\mathbb E}(|N^{\pm}_{F_0}(v) \cap V_i|) = (1-o(1))\frac{p'}{p} \frac {|V_i|}{n}
    |N^{\pm}_D(v)| = (1 \pm \lambda \pm o(1))|V_i|p' = (1 \pm 2\lambda )|V_i|p',$$
    But then by Chernoff's inequality
    we have that
    \begin{equation}
        \label{equation: degree control in the reduced graph}
        {\Pr}(|N^{\pm}_{F_0}(v) \cap V_i| \notin
        (1 \pm 3\lambda )|V_i|p') \leq 2 e^{-\frac {\lambda ^2 |V_i|p'}{3}}
        \leq 2e^{-\frac {\lambda ^2 |V_i|p'}{4}}
        \leq 2e^{- \frac{\lambda ^2 \log ^2n}{4}} = o(1/n^3).
    \end{equation}
    The second last inequality holds since $|V_i|p' \geq mp' =
    \frac {sp'}{\log n} = \sqrt {\frac {np}{\alpha \log ^6n} }\cdot
    \frac {1}{\log n} \gg \sqrt \frac {\log ^{14} n}{\alpha \log ^6 n}
    \frac {1}{\log n} >\log^2n$.
    By
        \eqref{equation: degree control in the reduced graph}
        this gives
    that with probability $1 - o(1/n)$ we have  $d^{\pm}_{F_0}(v) =
    (1 \pm 3\lambda )(s+m)p'$ for all $v\in V(F_0)$, as required.

    To see (B), note that for any ${\mathcal M}$, each set $X \subset V(F_0)$
    corresponds to a set $X^* \subset V(D)$ with $|X^*| \leq 2 |X|$, obtained
    by `opening the pairs of $X$', i.e. $X^* = (X\setminus \mathcal M)
    \cup \{s_i,t_i: (s_i,t_i) \in X\}$. Thus to prove (B) it suffices to show
    that with probability $1 - o(1/n)$, every set $X^* \subset V(D)$ with
    $|X^*| \leq \frac {2\log ^2(s+m)}{p'}$ satisfies $e_{F_0}(X^*)
    \leq \frac {|X^*| \log ^{2.1}n}{2}$.

    Now for $|X^*| \leq \frac {2 \log ^2(s+m)}{p'}$ since $p' = (1 \pm o(1))
    \frac{p}{ \log ^6n}$, we have $|X^*| \leq \frac {4 \log ^8 n}{p}$.
    From Definition
    \ref{definition:pseudorandom} {\textbf{(P2)}} we have
    $$e_{D}(X^*)
    \leq (1-\lambda)|X^*|\log^{8.02}n\leq \frac{(1-\lambda)|X^*|
    \log^{8.05}n}{2}.$$
    We now want to estimate $e_{F_0}(X^*)$. Since in $F_0$, each edge from $D$ is included independently
    with probability $(1 \pm o(1))p'/p = \frac{(1 \pm o(1))}{\log ^6 n}$.
    By Lemma \ref{Che}, since
    ${\mathbb E}(e_{F_0}(X^*)) = (1 \pm o(1))e_{D }(X^*)p'/p
    \leq |X^*|\log ^{2.05}n /2$ we have that
        \begin{align*}
            {\Pr}\left(e_{F_0}(X^*) > \frac{|X^*| \log ^{2.1}(s+m)}{2}\right)
            & \leq
            \left(\frac{2e\cdot  e_{D}(X^*)p'/p}{|X| \log ^{2.1}(s+m)}\right)^{|X| \log ^{2.1}(s+m)} \\
       & \leq \left(\frac{40}{\log ^{0.05} n}\right)^{|X| \log ^{2.1}(s+m)}
        \end{align*}
   The final inequality here holds as $\log (s + m) \geq \log n/3$ for
   $s \geq \sqrt n/\alpha (n) \geq n^{1/3}$.
   But there are $\binom {n}{x} \leq e^{x \log n}$ sets of size $|X^*| = x$.
   Therefore
   \begin{equation*}
        {\Pr}\left(e_{F_0}(X^*) > |X^*| \log ^{2.1}(s+m)/2 \mbox{ for some } X^*\right) \leq \sum _{x = 1}^{m +s} e^{x \log n} \Big (\frac{40}{\log ^{0.05} n} \Big )^{x \log ^{2.1}(s+m)}
        = O \Big (\frac{1}{n^2} \Big ).
   \end{equation*}
   This completes the proof of (B).

    To prove $(C)$, first note that by averaging it suffices to prove
    $(C)$ when $X,Y\subseteq V(F_0)$ are two disjoint subsets
    with $|X|,|Y| = k= \lceil \frac{\log^{1.1} n}{p'} \rceil $.
    Given any choice of ${\mathcal M}$ and such sets $X$ and $Y$, let
        $X^*=\left(X\setminus \M\right)\cup\{t_i\mid \ (s_i,t_i)\in X\}$ and
        $Y^*=\left(Y\setminus \M\right)\cup\{s_i\mid\ (s_i,t_i)\in Y\}$.
    Note that $|X|= |X^*|$ and $|Y|=|Y^*|$ and
    from \textbf{(P3)} of Definition \ref{definition:pseudorandom} we have
    $e_{F_0}(X,Y)= e_{F_0}(X^*,Y^*)$. Thus to prove $(C)$ for all
    choices of ${\mathcal M}$, it suffices
    to prove that with probability $1 - o(1/n)$, we have
    $e_{F_0}(X^*,Y^*) \leq (1+ 2\lambda)|X^*||Y^*|p'$ for all disjoint sets
    $X^*, Y^* \subset V({F_0})$ with $|X^*| = |Y^*| = k$.

    To see this, note that for such $X^*, Y^*$, from Definition
    \ref{definition:pseudorandom} {\textbf(P3)} we have
    $e_{D} (X^*,Y^*) = (1 \pm \lambda )|X^*||Y^*|p$. This gives that
    ${\mathbb E}(e_{F_0 }(X^*,Y^*)) =
    (1 \pm \lambda \pm o(1))|X^*||Y^*|p'$ and by Chernoff's inequality we find
    $${\Pr}\left(e_{F_0 }(X^*,Y^*)\right) >
    (1 + 2\lambda )|X^*||Y^*|p') \leq e^{ - \frac{\lambda ^{2}|X^*||Y^*|p'}{6}} = e^{- \lambda ^2 k^2 p'/6}.$$
    Thus the probability that $e_{F_0}(X^*,Y^*)>(1 + 2\lambda)|X||Y|p'$ for some
    such pair is at most
    \begin{align*}
        \binom {n}{k}^2 e^{-\frac {\lambda ^2 k^2 p'}{6}}
        \leq (n e^{-\frac {\lambda ^{2} kp'}{6}})^k = o(1/n),
    \end{align*}
    where the final equality holds by choice of $k$. This completes the proof
    (C).
\end{proof}

We now give the proof of the Lemma \ref{lemma:bipartite is pseudorandom} stated in Section \ref{sec:PropPseudo}.

\begin{proof}[Proof of Lemma \ref{lemma:bipartite is pseudorandom}.]
    To prove the lemma, we will first show that with probability $1-o(1/n)$, for every $j\in [\ell -1]$ the digraph $F_j$ satisfies the following properties:
    \begin{enumerate}[$(i)$]
        \item $e_{F_j}(X,Y)=
        (1\pm 2\lambda)|X||Y|p\cdot p_{in}$ for every two subsets $X\subseteq V_j$ and $Y\subseteq V_{j+1}$ with
        $ |X|,|Y|\geq k = \lceil \frac{24\log n}{\lambda ^{2} pp_{in}} \rceil $,
        \item $e_{F_j }(X,Y) \leq \min\{|X|, |Y|\}\log ^{2.05}n$ for all $X \subset V_j$ and $Y \subset V_{j+1}$
        with $|X|, |Y| \leq k$,
        \item $d^{\pm}(v,V_{j+1})\geq (1-2\lambda)mp\cdot p_{in}$ for every $v\in V_j$.
    \end{enumerate}
    We first prove (\emph{i}). First note that by an easy averaging argument, it suffices to prove this for all such sets $X$ and $Y$ with $|X| = |Y| = k$. Now as $D$ is  $(n,\lambda,p)$ pseudo-random and $k \geq \log ^{1.1}n/p$, from property \textbf{(P3)} of Definiton
    \ref{definition:pseudorandom} we have $e_{D}(X,Y)=
    (1\pm \lambda)|X||Y|p$ for every such $X$ and $Y$. Let $N_{X,Y}$ be the number of edges in $F_j [X,Y]$. Then $N_{X,Y}\sim \Bin(e_{D}(X,Y),p_{in})$ and thus ${\mathbb E}(N_{X,Y}) = e_{D}(X,Y)p_{in} \geq (1-\lambda )|X||Y|p\cdot p_{in}$. By Chernoff's inequality,
    $$   \Pr \left(N_{X,Y}\notin (1\pm\lambda) e_{D}(X,Y)p_{in}\right)\leq e^{-\frac {\lambda^2}3(1- \lambda)|X||Y|p\cdot p_{in}}\leq e^{-\frac {\lambda^2k^2pp_{in}}{6}}.$$
    By a union bound, this gives that
    \begin{align*} \Pr \left(N_{X,Y}\notin (1\pm\lambda) e_{D}(X,Y)p_{in} \mbox{ for some pair } X \mbox{ and } Y \right)& \leq \binom {m}{k}^2 e^{-\frac {\lambda^2}3(1- \lambda)|X||Y|p\cdot p_{in}} \leq n^{2k}e^{-\frac {\lambda^2k^2pp_{in}}{6}}\\
    & = (ne^{-\frac{\lambda ^2 k p p_{in}}{12}})^{2k} = o(1/n).
    \end{align*}
    The final equality here holds by the definition of $k$.

    Property (\emph{ii}) holds immediately from property \textbf{(P2)} in Definition \ref{definition:pseudorandom}.

    We now show (\emph{iii}). From \textbf{(P1)} of Definition \ref{definition:pseudorandom} we have that $d^{\pm}_D(v)=(1\pm\lambda)np$. Since for each $j\in [\ell-1]$ the set $V_{j+1}$ is chosen uniformly at random, the degree of $v$ in $V_{j+1}$ is distributed according to the hypergeometric distribution with parameters $n,\ d^{\pm}_D(v) ,\ |V_{j+1}|$. By Chernoff's inequality we have
    $$\Pr \left( d^{\pm }_D(v,V_{j+1})<(1-\lambda/2)\frac {d^{\pm }_D(v)m}{n}\right)\leq
    e^{-\frac{\lambda ^2 d^{\pm}_{\D}(v) m}{6n}} \leq
    e^{-\frac {\lambda^2}{12}mp}=o\left(1/n^3\right).$$
    Therefore, with probability $1-o(1/n)$ we have that for every $j\in[\ell-1]$ and $v\in V_j$ we have $d^{\pm}_D(v,V_{j+1})\geq (1-\lambda/2)mp$. We now use this to estimate $d^{\pm}_{F_j}(v)$ where $v\in V_j$. As in $F$ every edge appears independently with probability $p_{in}$, by Chernoff's inequality we have
    $$\Pr\left(d^{\pm}_{F}(v,V_{j+1})<(1-2\lambda)mp p_{in}\right)
    \leq e^{-\frac {\lambda ^2 mp\cdot p_{in}}{6}} =o(1/n^2).$$
    Thus with probability $1-o(1/n)$ we have that $d^{\pm}_{F}(v,V_{j+1})\geq(1-2\lambda)mp\cdot p_{in}$, i.e. (\emph{iii}) holds. \\

Using (\emph{i}), (\emph{ii}) and (\emph{iii}) we can now complete the proof of the lemma.
It suffices to show that $F_j$ contains an $r$-regular subgraph, where $r = (1-4 \lambda )mp\cdot p_{in}$. To see this, by the Gale-Ryser theorem, it suffices to show that for all $X \subset V_{j}$ and $Y \subset V_{j+1}$ we have
\begin{equation}
\label{equation: Gale-Ryser condition} e_{F }(X,Y) \geq r(|X| + |Y| - m).
\end{equation}
Suppose that $|X| = x$ and $|Y| = y$. It clearly suffices to work with the case when $x + y \geq m$. First note that if $x, y \geq k$ then by (\emph{i}) we have
$$
e_{F}(X,Y) \geq (1 - 2\lambda )xyp\cdot p_{in} \geq (1-2\lambda )m(x+y - m)p \cdot p_{in}
> r(x + y - m)$$
The second last inequality here holds since $(m-x)(m-y) \geq 0$. It remains to prove that (\ref{equation: Gale-Ryser condition}) holds for $X, Y$ satisfying $x + y \geq m$ with either $x \leq k$ or $y\leq k$. We will prove this for $x \leq k$, as the other case is identical. Since $x+y \geq m$, we have $y \geq m - x $. But then $ |Y^c| \leq |X| \leq k$ and
\begin{align*}
    e_{F}(X,Y) = e_{F}(X,V_{j+1}) - e_{F}(X,Y^c)
    & \geq (1- 2\lambda )xmp\cdot p_{in}- (|X|+ |Y^c|)\log ^{2.05}n\\
    & \geq x(1-2\lambda )mp\cdot p_{in} - 2x\log ^{2.05}n\\
    & = x(1-4 \lambda )mp\cdot p_{in} + x(2\lambda mp\cdot p_{in} - 2 \log ^{2.05} n)\\
    & \geq x(1-4 \lambda )mp \cdot p_{in} \geq  r(x + y - m).
\end{align*}
The first inequality here holds by (\emph{ii}) and (\emph{iii}) and the third inequality holds since $\lambda mp \cdot p_{in}=\omega(log ^{2.05} n)$ (note that this is true provided $p$ is a sufficiently large power of $\log n$).
\end{proof}

\end{document}